\documentclass[a4paper]{amsart}

\usepackage{amsfonts}
\usepackage{amssymb}

\usepackage[T1]{fontenc}

\usepackage{enumerate}

\usepackage[abbrev]{amsrefs}

\usepackage[utf8]{inputenc}

\usepackage{backref}

\newcommand{\abs}[1]{\mathopen\lvert#1\mathclose\rvert}
\newcommand{\bigabs}[1]{\bigl\lvert#1\bigr\rvert}
\newcommand{\norm}[1]{\mathopen\lVert#1\mathclose\rVert}
\newcommand{\bignorm}[1]{\mathopen\big\lVert#1\mathclose\big\rVert}
\newcommand{\floor}[1]{\lfloor#1\rfloor}
\newcommand{\N}{{\mathbb N}}
\newcommand{\R}{{\mathbb R}}
\renewcommand{\S}{{\mathbb S}}
\newcommand{\cH}{\mathcal{H}}
\newcommand{\cM}{\mathcal{M}}
\newcommand{\cQ}{\mathcal{Q}}
\newcommand{\cT}{\mathcal{T}}
\DeclareMathOperator{\supp}{supp}
\DeclareMathOperator{\dist}{dist}
\DeclareMathOperator{\sgn}{sgn}
\newcommand{\dif}{\,\mathrm{d}}

\theoremstyle{plain}
\newtheorem{proposition}{Proposition}[section]
\newtheorem{lemma}[proposition]{Lemma}
\newtheorem{theorem}{Theorem}

\theoremstyle{remark}

\newtheorem{openproblem}{Open Problem}
\newtheorem*{Claim}{Claim}

\usepackage{refcount}
\newcounter{cte}
\newcommand{\Constant}{\refstepcounter{cte} C_{\thecte}}
\newcommand{\NewConstant}{\setcounter{cte}{1} C_{\thecte}}
\newcommand{\SameConstant}{C_{\thecte}}

\numberwithin{equation}{section}

\title[Density for fractional Sobolev spaces into manifolds]{Density of smooth maps for fractional Sobolev spaces $W^{s, p}$ into $\ell$ simply connected manifolds when $s \ge 1$}

\author{Pierre Bousquet}
\address{
Pierre Bousquet\hfill\break\indent
Aix-Marseille Université, CMI 39\hfill\break\indent
Laboratoire d'analyse, topologie, probabilités UMR7353\hfill\break\indent
Rue Fr\'ed\'eric Joliot Curie\hfill\break\indent
13453 Marseille Cedex 13\hfill\break\indent
France}

\email{bousquet@cmi.univ-mrs.fr}

\author{Augusto C. Ponce}
\address{
Augusto C. Ponce\hfill\break\indent
 Université catholique de Louvain\hfill\break\indent
 Institut de Recherche en Math{\'e}matique et Physique\hfill\break\indent
 Chemin du cyclotron 2, bte L7.01.02\hfill\break\indent
1348 Louvain-la-Neuve\hfill\break\indent
Belgium}
\email{Augusto.Ponce@uclouvain.be}

\author{Jean Van Schaftingen}
\address{
Jean Van Schaftingen\hfill\break\indent
 Université catholique de Louvain\hfill\break\indent
 Institut de Recherche en Math{\'e}matique et Physique\hfill\break\indent
 Chemin du cyclotron 2, bte L7.01.01\hfill\break\indent
1348 Louvain-la-Neuve\hfill\break\indent
Belgium}
\email{Jean.VanSchaftingen@uclouvain.be}

\begin{document}

\begin{abstract}
Given a compact manifold \(N^n \subset \R^\nu\), \(s \ge 1\) and \(1 \le p < \infty\), we prove that the class \(C^\infty(\overline{Q}^m; N^n)\) of smooth maps on the cube with values into \(N^n\) is strongly dense in the fractional Sobolev space \(W^{s, p}(Q^m; N^n)\) when \(N^n\) is \(\floor{sp}\) simply connected. 
For \(sp\) integer, we prove weak density of \(C^\infty(\overline{Q}^m; N^n)\) when \(N^n\) is \(sp - 1\) simply connected.
The proofs are based on the existence of a retraction of \(\R^\nu\) onto \(N^n\) except for a small subset of \(N^n\) and on a pointwise estimate of fractional derivatives of composition of maps in \(W^{s, p} \cap W^{1, sp}\).
\end{abstract}


\maketitle

\section{Introduction}

In this paper we discuss results and open questions related to the density of smooth maps in Sobolev spaces with values into a manifold. 
For this purpose, let $N^n$ be a compact manifold of dimension \(n\) imbedded in the Euclidean space \(\R^\nu\). 
For any  \(s>0\) and \(1 \le p < +\infty\), we define the class of Sobolev maps defined on the unit \(m\) dimensional cube \(Q^m\) with values into \(N^n\),
\[
W^{s, p}(Q^m; N^n) 
= \big\{u \in W^{s, p}(Q^m; \R^\nu) : u \in N^n \ \textrm{a.e.} \big\}.
\]
When \(s = k\) is an integer, \(W^{s, p}(Q^m; \R^\nu)\) is the standard Sobolev space equipped with the norm
\[
\norm{u}_{W^{s, p}(Q^m)} = \norm{u}_{L^p(Q^m)} + \sum_{j = 1}^k \norm{D^j u}_{L^p(Q^m)}.
\]
When \(s\) is not an integer, \(s = k + \sigma\) with \(k \in \N\) and \(0 < \sigma < 1\).
In this case, by \(u \in W^{s, p}(Q^m; \R^\nu)\) we mean that \(u \in W^{k, p}(Q^m; \R^\nu)\) and
\[
[D^k u]_{W^{\sigma, p}(Q^m)} 
= \bigg( \int\limits_{Q^m}\int\limits_{Q^m} \frac{|D^{k} u(x)-D^{k} u(y)|^p}{|x-y|^{m+\sigma p}} \dif x \dif y \bigg)^{1/p} < +\infty,
\]
and the associated norm is given by
\[
\norm{u}_{W^{s, p}(Q^m)} = \norm{u}_{L^p(Q^m)} + \sum_{j = 1}^k \norm{D^j u}_{L^p(Q^m)} + [D^k u]_{W^{\sigma, p}(Q^m)}.
\]

The fractional Sobolev spaces \(W^{s, p}(Q^m; \R^\nu)\) arise in the trace theory of Sobolev spaces of integer order. 
For example, the trace is a continuous linear operator from \(W^{1, p}(Q^m; \R^\nu)\) onto \(W^{1 - \frac{1}{p}, p}(\partial Q^m; \R^\nu)\) \cite{Gagliardo}*{Theorem~1.I}. 
\medskip

We first address the question of \emph{strong density} of smooth maps: given $u \in W^{s, p}(Q^m; N^n)$, does there exist a sequence in $C^{\infty}(\overline Q^m; N^n)$ which converges to \(u\) with respect to the strong topology induced by the \(W^{s, p}\) norm?

A naive approach consists in applying a standard regularization argument.
This works well for maps in $W^{s, p}(Q^m; \R^{\nu})$ and
shows that $C^{\infty}(\overline Q^m; \R^\nu)$ is strongly dense in that space.
When $\R^{\nu}$ is replaced by $N^n$, the conclusion is less clear since the convolution of a map $u \in W^{s, p}(Q^m; N^n)$ with a smooth kernel $\varphi_{t}$ yields a map with values in the convex hull of $N^n$.
In this case, one might try to project \(\varphi_t * u\) into the manifold \(N^n\).
This is indeed possible for \(sp \ge m\).

If $sp>m$, then by the Morrey-Sobolev imbedding, $W^{s, p}$ is continuously imbedded into \(C^0\).
Thus, every map $u\in W^{s,p}(Q^m; N^n)$ has a continuous representative and \(\varphi_t * u\) converges uniformly as \(t\) tends to zero. 
In particular,
\begin{equation}\label{contdist}
\lim_{t \to 0}\sup_{x\in Q^m}{\dist{(\varphi_{t} * u(x), N^n)}} =  0.
\end{equation}
Hence, one may project $\varphi_{t} * u$ back to $N^n$ since the nearest point projection \(\Pi\) is well defined and smooth on a neighborhood of $N^n$.

If \(sp = m\), then the Morrey-Sobolev imbedding fails but property \eqref{contdist} remains true since \(W^{s, p}\) injects continuously into the space VMO of functions with vanishing mean oscillation.
This fact has been observed by Schoen and Uhlenbeck~\cite{SchoenUhlenbeck}.
 
We may summarize as follows:

\begin{theorem} \label{premiertheorem}
If $sp\geq m$, then $C^{\infty}(\overline Q^m; N^n)$ is strongly dense in $W^{s,p}(Q^m; N^n)$.
\end{theorem}

The case where $sp<m$ is more subtle and the answer depends on the topology of \(N^n\). 
Even when $N^n$ is the unit sphere \(\S^n\) the approximation problem is not fully understood. 
For instance, consider the map \(u : B^3 \to \S^2\) defined by
\[
u(x) = \frac{x}{|x|}.
\] 
Then, $u\in W^{s, p}(B^3; \S^2)$ for every \(s > 0\) and \(p \ge 1\) such that \(sp < 3\), but $u$ cannot be strongly approximated in $W^{s, p}$ by smooth maps with values into $\S^2$ when \(2 \le sp < 3\). 
This example originally due to Schoen and Uhlenbeck \cite{SchoenUhlenbeck} for \(s = 1\) can be adapted to the case where \(\S^2\) is replaced by any compact manifold \(N^n\) and for any value of \(s\) \citelist{\cite{Escobedo}*{Theorem~3} \cite{Mironescu}*{Theorem~4.4}}.

\begin{theorem}
\label{theoremNonDensity}
If \(sp < m\) and \(\pi_{\floor{sp}}(N^n) \not= \{0\}\), then $C^{\infty}(\overline Q^m; N^n)$ is not strongly dense in $W^{s, p}(Q^m; N^n)$.
\end{theorem}

It seems that the topological condition \(\pi_{\floor{sp}}(N^n) \not= \{0\}\) is the only obstruction to the strong density of smooth maps in $W^{s, p}(Q^m; N^n)$.
This is indeed true when \(s\) is an integer by a remarkable result of Bethuel~\citelist{\cite{Bethuel}*{Theorem~1} \cite{Hang-Lin}} for \(s=1\) which has been recently generalized by the authors \cite{Bousquet-Ponce-VanSchaftingen}*{Theorem~4} for any \(s \in \N\) (see also \cite{Gastel-Nerf}):

\begin{theorem}
\label{theoremStrongDensityInteger}
Let \(s \in \N_*\).
If  \(sp < m\) and \(\pi_{\floor{sp}}(N^n) = \{0\}\), then $C^{\infty}(\overline Q^m; N^n)$ is  strongly dense in $W^{s, p}(Q^m; N^n)$.
\end{theorem} 

Some cases of non-integer values have been investigated.
For instance when \(s=1-1/p\) in the setting of trace spaces \cites{Bethuel-1995, Mucci} and also when \(s \geq 1\) and \(N^n = \S^n\) \cites{Escobedo, Bethuel-Zheng}.
Brezis and Mironescu~\cite{Brezis-Mironescu} have announced in a personal communication a solution to the question of strong density for any \(0 < s < 1\).

All these cases give an affirmative answer to the following:

\begin{openproblem}
\label{OP1}
Let \(s \not\in \N_*\). 
If \(sp < m\) and \(\pi_{\floor{sp}}(N^n)= \{0\}\), is it true that \(C^{\infty}(\overline Q^m ;  N^n)\) is strongly dense in \(W^{s, p}(Q^m; N^n)\)?
\end{openproblem}

In this paper, we investigate Open~problem~\ref{OP1} for \(\ell\) simply connected manifolds \(N^n\):
\begin{equation}\label{conditiontopologique}
\pi_0(N^n) = \dots = \pi_{\ell}(N^n) = \{0\}.
\end{equation}

We prove the following:

\begin{theorem}
\label{deuxiemetheorem}
Let \(s \ge 1\). 
If \(sp < m\) and if \(N^n\) is \(\floor{sp}\) simply connected, then \(C^\infty(\overline Q^m; N^n)\) is strongly dense in \(W^{s, p}(Q^m; N^n)\).
\end{theorem}

Even in the case where \(s\) is an integer --- which is covered in full generality by Theorem~\ref{theoremStrongDensityInteger} --- the proof
is simpler and has its own interest.
We have been inspired by Haj{\l}asz~\cite{Hajlasz} who has proved Theorem~\ref{deuxiemetheorem} for \(s=1\).

Our proof of Theorem~\ref{deuxiemetheorem} is based on two main ingredients.
The \emph{geometric tool} (Proposition~\ref{propositionSmoothProjection}) gives a smooth retraction of the ambient space \(\R^\nu\) onto \(N^n\) except for a small subset of \(N^n\).
The \emph{analytic tool} (Proposition~\ref{propositionPointwiseEstimateWsp}) gives a pointwise estimate of the fractional derivative of \(\eta \circ u\), where \(\eta\) is a smooth map and \(u\) is a \(W^{s, p}\) map.

The counterpart of Theorem~\ref{deuxiemetheorem} for \(0 < s <1\) requires different tools and will be investigated in a subsequent paper.

\medskip

The second problem we adress in this paper concerns the \emph{weak density} of \(C^{\infty}(\overline Q^m; N^n)\)  in \(W^{s, p}(Q^m; N^n)\):
given \( u \in W^{s, p}(Q^m; N^n)\), does there exist a sequence in  \(C^{\infty}(\overline Q^m; N^n)\) which is bounded in  \(W^{s, p}(Q^m; N^n)\) and converges to \(u\) in measure?

The case \(sp \ge m\) has an affirmative answer due to the strong density of smooth maps.
When \(sp < m\), we find the same topological obstruction as for the strong density problem when \(sp\) is not an integer \cite{Bethuel}*{Theorem~3}:

\begin{theorem}
\label{theoremWeakDensityNecessary}
If \(sp < m\) is such that \(sp \not\in \N\) and if \(C^{\infty}(\overline Q^m; N^n)\) is weakly dense in \(W^{s, p}(Q^m; N^n)\), then \(\pi_{\floor{sp}}(N^n) = \{0\}\).
\end{theorem}

From Theorem~\ref{theoremStrongDensityInteger}, it follows that for every \(s \in \N_*\) such that \(sp \not\in \N\) the problems of weak and strong density of smooth maps in \(W^{s, p}(Q^m; N^n)\) are equivalent.
We expect the same is true for \(s \not\in \N\).

The conclusion of Theorem~\ref{theoremWeakDensityNecessary} need not be true when \(sp\) is an integer. 
For instance, by a result of Bethuel~\cite{Bethuel-1990}*{Theorem~3}, \(C^{\infty}(\overline Q^3 ; \S^2)\) is weakly dense in \(W^{1,2}(Q^3 ; \S^2)\), even though it is not strongly dense by Theorem~\ref{theoremNonDensity}.
 
As a byproduct of the tools we use to prove Theorem~\ref{deuxiemetheorem}, we establish the following:

\begin{theorem}\label{theoremweakdensity}
Let \(s \geq 1\). 
If  \(sp < m\) is such that \(sp \in \N\) and if \(N^n\) is \(sp-1\) simply connected, then \(C^{\infty}(\overline Q^m; N^n)\) is weakly dense in \(W^{s, p}(Q^m; N^n)\).
\end{theorem}

This result is due to Haj\l asz~\cite{Hajlasz}*{Corollary~1} when \(s=1\); Haj\l asz's argument still applies for \(p=1\) although it is not explicitly stated in his paper.
More recently, Hang and Lin~\cite{Hang-Lin-2003}*{Corollary~8.6} proved an analogue of Theorem~\ref{theoremweakdensity} under a weaker topological assumption for \(s = 1\).
To our knowledge, the only result concerning weak density of smooth maps for non-integer values of \(s\) deals with the case
\(s=1/2\), \(p=2\) and \(N=\S^1\) and is due to Rivière~\cite{Riviere}*{Theorem~1.2}.

Combining Theorem~\ref{theoremNonDensity} and Theorem~\ref{theoremweakdensity} we deduce that
\(C^{\infty}(\overline Q^m ; \S^n)\) is weakly dense but not strongly dense in \( W^{s,p}(Q^m ; \S^n)\) for \(p \le m\) and \(sp = n\).

When \(sp \in \N\),  we do not know whether  \(C^{\infty}(\overline Q^m; N^n)\) is weakly dense in \(W^{s, p}(Q^m; N^n)\) with no assumption on \(N^n\). 
The only results which are known in this sense concern \(s=1\): for \(p = 1\) \citelist{\cite{Hang}*{Theorem~1.1} \cite{Pakzad}*{Theorem~I} } and \(p=2\) \cite{Pakzad-Riviere}*{Theorem~I}.



\section{Main tools}

\subsection{Geometric tool}
Our first tool is the construction of a retraction of \(\R^\nu\) onto \(N^n\) except for a small subset of \(N^n\). 
This is the only place where the topological assumption \eqref{conditiontopologique} concerning the \(\ell\) simply connectedness of the manifold \(N^n\) comes into place.

\begin{proposition}
\label{propositionSmoothProjection}
If \(N^n\) is \(\ell\) simply connected, then for every \(0 < \epsilon \le 1\) there exist a smooth function \(\eta : \R^\nu \to N^n\) and a compact set \(K \subset N^n\) such that
\begin{enumerate}[\((i)\)]
\item for every \(x \in N^n \setminus K\), \(\eta(x) = x\),
\label{1655}
\item \(\cH^n(K) \le C \epsilon^{\ell + 1}\), for some constant \(C > 0\) depending on \(N^n\) and \(\nu\),
\label{1656}
\item for every \(j \in \N_*\),
\[
\norm{D^j \eta}_{L^\infty(\R^\nu)} \le \frac{C'}{\epsilon^j},
\]
for some constant \(C' > 0\)  depending on \(N^n\), \(\nu\) and \(j\).
\label{1657}
\end{enumerate}
\end{proposition}
 
The set \(K\) is chosen as the \(\epsilon\) neighborhood of an \(n - \ell - 1\) dimensional dual skeleton of \(N^n\). 
This proposition is the smooth counterpart of Haj\l asz's construction of a Lipschitz continuous map \(\eta\)~\cite{Hajlasz}*{Section~4}.

The proof of Proposition~\ref{propositionSmoothProjection} relies on the existence of a triangulation of the manifold \(N^n\).
It is more convenient to use a variant of the triangulation based on the decomposition of \(N^n\) in terms of cubes rather than simplices.

A cubication \(\cT\) of \(N^n\) is a finite collection of closed sets covering \(N^n\) of the form \(\Phi(\sigma)\) with \(\sigma \in \cQ\) such that 
\begin{enumerate}[\((a)\)]
\item \(\Phi : \bigcup\limits_{\sigma \in \cQ} \sigma \to N^n\) is a biLipschitz map,
\item \(\cQ\) is a finite collection of cubes of dimension \(m\) in some Euclidean space \(\R^\mu\), such that two elements of \(\cQ\) are either disjoint or intersect along a common face of dimension \(\ell\) for some \(\ell \in \{0, \dots, n\}\). 
\end{enumerate}
Given \(\ell \in \{0, \dots, n\}\), we denote by \(T^\ell\) the union of all \(\ell\) dimensional faces of elements of \(\cT\); 
we call \(T^\ell\) the \(\ell\) dimensional skeleton of \(\cT\).

We recall the following lemma~\cite{Hajlasz}*{Lemma~1}:

\begin{lemma}
\label{lemmaHajlaszContinuous}
Let \(\cT\) be a cubication of \(N^n\) and let \(T^\ell\) be the \(\ell\) dimensional skeleton of \(\cT\).
If \(N^n\) is \(\ell\) simply connected,  then there exists a Lipschitz continuous function \(\underline{\eta} : \R^\nu \to N^n\) such that for every \(x \in T^\ell\), \(\underline{\eta}(x) = x\).
\end{lemma}

\begin{proof}
Let \(CT^\ell \subset \R \times \R^\nu\) denote the cone
\[
\bigl\{ (\lambda, \lambda x) \in \R \times \R^\nu :  \lambda \in [0, 1] \text{ and }x \in T^\ell \bigr\}.
\]
Since \(CT^\ell\) is contractible, there exists a continuous map \(\xi : \R^\nu \to CT^\ell\) such that for every \(x \in T^\ell\), \(\xi(x) = (1, x)\).

We may choose \(\xi\) to be uniformly continuous.
Indeed, if \(p : \R^\nu \to \R^\nu\) is any Lipschitz function such that \(p \) coincides with the identity on \(T^\ell\) and \(p\) is constant outside some ball containing \(T^\ell\), then for every \(x \in T^\ell\), \(\xi \circ p(x) = (1, x)\) and, in addition, \(\xi \circ p\) is uniformly continuous.
Replacing \(\xi\) by \(\xi \circ p\) if necessary, we assume in the sequel that \(\xi\) itself is uniformly continuous.

Since \(N^n\) is \(\ell\) simply connected, the identity map in \(N^n\) is homotopic to a continuous map in \(N^n\) which is constant on \(T^\ell\) \cite{White}*{Section~6}. 
More precisely, there exist a continuous map \(H : [0, 1] \times N^n \to N^n\) and \(a \in N^n\) such that
\begin{enumerate}[\((a)\)]
\item for every \(x \in T^\ell\), \(H(0, x) = a\),
\item for every \(x \in N^n\), \(H(1, x) = x\).
\end{enumerate}
Since \(H\) is constant on \(\{0\} \times T^\ell\), \(H\) induces a continuous quotient map \(\overline H : CT^\ell \to N^n\) defined for every \((\lambda, \lambda x) \in CT^\ell\) by \(\overline H(\lambda, \lambda x) = H(\lambda, x)\).
Then, \(\overline{H} \circ \xi\) is a uniformly continuous map with values into \(N^n\) which coincides with the identity map on \(T^\ell\).

Using a standard approximation argument, we may construct a Lipschitz map having the same properties. 
We present the argument for the sake of completeness.

Given \(\iota > 0\), let \(\theta : \R^\nu \to [0, 1]\) be a Lipschitz continuous function supported in a neighborhood of \(T^\ell\) such that 
\begin{enumerate}[\((a')\)]
\item for every \(x \in T^\ell\), \(\theta(x) = 1\),
\item for every \(x \in \supp{\theta}\), \(\abs{x -\overline H \circ \xi(x)} \le \iota\).
\end{enumerate}
Since \(\overline H \circ \xi\) is uniformly continuous, there exists a Lipschitz approximation  \(h: \R^\nu \to \R^\nu\) such that for every \(x \in \R^\nu\),
\[
\abs{h(x) - \overline H \circ \xi(x)} \le \iota.
\]
Then, for every \(x \in \R^\nu\),
\[
\bigabs{\overline H \circ \xi(x) - \big( \theta(x) x + (1 - \theta(x)) h(x) \big)} \le \iota.
\]
Since \(\overline H \circ \xi(x) \in N^n\), it follows that 
\[
\theta(x) x + (1 - \theta(x)) h(x) \in N^n + \overline B^\nu_{\iota},
\]
where \(\overline B_\iota^\nu\) is the closed ball in \(\R^\nu\) of radius \(\iota\) centered at \(0\).
Choosing \(\iota\) such that the nearest point projection \(\Pi : N^n + \overline B^\nu_{\iota} \to N^n\) is well-defined and smooth, then we have the conclusion by taking \(\underline{\eta} : \R^\nu \to N^n\) defined for \(x \in \R^\nu\) by
\[
\underline{\eta}(x)
= \Pi\big( \theta(x) x + (1 - \theta(x)) h(x) \big).
\]
The proof is complete.
\end{proof}

We shall also use dual skeletons associated to a cubication \(\cT \) given by a map \(\Phi : \bigcup\limits_{\sigma \in \cQ} \sigma \to N^n.\)  
We  first define dual skeletons  for a cube in \(\R^n\). 
Let  \(j \in \{0, \dots, n\}\). 
When the center of the cube is \(0\) and the faces are parallel to the coordinate axes,  the dual skeleton of dimension \(j\) is the set of points in the cube which have at least \(n - j\) components equal to zero. 
By using an isometry, we can define the dual skeleton of a cube of dimension \(n\) in \(\R^\mu\) in general position. 
Then, the dual skeleton of dimension \(j\) of a family  \(\cQ\) of cubes   as  above is simply the union of the dual skeletons  of dimension \(j\) of each cube. Finally, the dual skeleton \(L^j\) of dimension \(j\)  of the  cubication \(\cT\) of \(N^n\) is the image by \(\Phi\) of the \(j\) dimensional  dual skeleton of \(\cQ\).
 
The following lemma implies the homotopy equivalence between the skeleton \(T^\ell\) of the manifold \(N^n\) and the complement of the dual skeleton \(L^{n - \ell - 1}\) in \(N^n\).
We are particularly interested in the pointwise estimates of the homotopy \(f\):

\begin{lemma}\label{lemmaf}
Let \(\ell \in \{0, \dots, n-1\}\), let \(\cT\) be a cubication of \(N^n\) and let \(L^{n - \ell - 1}\) be the \(n - \ell - 1\) dimensional dual skeleton  of \(\cT\). 
Then, there exists a locally Lipschitz continuous function 
\[
f : [0, 1] \times (N^n \setminus L^{n - \ell - 1} ) \to N^n
\]
such that
\begin{enumerate}[\((i)\)]
\item for every \(t \in [0, 1]\) and for every \(x \in T^\ell\), \(f(t, x) = x\),
\item  for every \(x \in N^n \setminus L^{n - \ell - 1}\), \(f(0, x) = x\) and \(f(1, x) \in T^\ell\),
\item for every \(t \in [0, 1]\) and for every \(x \in N^n \setminus L^{n - \ell - 1}\),
\[
\abs{\partial_t f(t, x)} \le C,
\]
and
\[
\abs{\partial_x f(t, x)} \le \frac{C'}{\dist{(x, L^{n - \ell - 1})}},
\]
for some constants \(C, C' > 0\) depending on \(n\), \(\ell\), \(N^n\) and \(\cT.\) 
\end{enumerate}
\end{lemma}

\begin{proof}
We first establish the result when the manifold \(N^n\) is replaced by the cube \([-1, 1]^n\)  and  \(L^{n-\ell-1}\) is the dual skeleton of dimension \(n-\ell-1\) of \([-1, 1]^n\).
Following \cite{White}, we consider for every $x\in [-1, 1]^n$,
$$
|x|_\ell =\min_{\substack{S\subset \{1, \dots, n\}\\ |S|=\ell+1}}\max_{i\in S} |x_i|.
$$   
In particular, for every \(x \in [-1, 1]^n\), \(x \in L^{n - \ell - 1}\) if and only if \(\abs{x}_\ell = 0\).
The function $x \in [-1, 1]^n \mapsto |x|_\ell$ is Lipschitz continuous of constant $1$.

Let \(\phi_\ell : [-1, 1]^n \setminus L^{n - \ell - 1} \to T^\ell\) be defined for every $x \in [-1, 1]^n$ by 
\[
\phi_{\ell}(x) = (y_1, \dots, y_n),
\]
where
\[
y_i = 
\begin{cases}
\sgn{x_i}	& \text{if \(|x_i|\geq |x|_\ell\),}\\
x_i/|x|_\ell	& \text{if \(|x_i| < |x|_\ell\).}
\end{cases}
\]

The homotopy \(f : [0, 1] \times ([-1, 1]^n \setminus L^{n - \ell - 1} ) \to [-1, 1]^n\) defined by
\[
f(t, x) = (1-t)x + t \phi_\ell(x)
\]
has the required properties.

In order to prove the existence of the homotopy \(f\) for a general compact manifold \(N^n\), we perform the above construction in every cube of a given cubication  \(\Phi : \bigcup\limits_{\sigma \in \cQ} \sigma \to N^n\). If two cubes \(\sigma_1\) and \(\sigma_2\) in \(\cQ\) have a non empty intersection, then the corresponding maps \(\phi_{\ell, 1}\) and \(\phi_{\ell, 2}\) coincide on the common face \(\sigma_1 \cap \sigma_2\). Hence, we can glue together the locally Lipschitz continuous maps obtained for each cube so as to obtain a global map \(f_0\) which is defined on the entire collection of cubes in \(\cQ\). 
The conclusion follows by taking 
\[
f(t,x)=\Phi\big(f_0(t, \Phi^{-1}(x)) \big).
\qedhere
\]   
\end{proof}

We now prove a counterpart of Proposition~\ref{propositionSmoothProjection} for a Lipschitz continuous map \(\eta\):

\begin{lemma}
\label{lemmaHajlaszLipschitz}
Let \(\ell \in \{0, \dots, n-1\}\), let \(\cT\) be a cubication of \(N^n\) and \(L^{n - \ell - 1}\) be the  \(n - \ell - 1\) dimensional dual skeleton of \(\cT\) and let \(\iota > 0\) be such that the nearest point projection \(\Pi\) onto \(N^n\) is smooth on \(N^n + \overline B_{2\iota}^\nu\).
If \(N^n\) is \(\ell\) simply connected, then for every \(0 < \epsilon \le 1\) there exists a Lipschitz continuous map \(\eta : \R^\nu \to N^n\) such that
\begin{enumerate}[\((i)\)]
\item \(\eta = \Pi\) on \((N^n + B_\iota^\nu) \setminus \Pi^{-1}(L^{n - \ell - 1} + B_\epsilon^\nu)\),
\label{1648}
\item for every \(x \in \R^\nu\),
\[
\abs{D \eta(x)} \le \frac{C''}{\epsilon}.
\]
for some constant \(C'' > 0\)  depending on \(N^n\), \(\cT\) and \(\nu\).
\label{1649}
\end{enumerate}
\end{lemma}

\begin{proof}
Let \(f\) be the map given by  Lemma \ref{lemmaf}.
The extension
\[
 \Bar{f} : \big( \{0\} \times L^{n - \ell - 1} \big) \cup \big([0, 1] \times (N^n \setminus L^{n - \ell - 1} ) \big) \to N^n
\]
defined by
\[
 \Bar{f} (t, x) 
= \begin{cases}
    x & \text{if \(t = 0\)},\\
    f (t, x) & \text{if \(0 < t \le 1\),}
  \end{cases}
\]
is continuous.

Let \(\Pi\) be the nearest point projection onto \(N^n\) and denote by \(\overline \Pi : \R^\nu \to \R^\nu\) a smooth extension of \(\Pi\).
The image of \(\overline\Pi\) need not be contained in the manifold \(N^n\). 

Let \(\theta : \R^\nu \to [0, 2]\) be a Lipschitz continuous function such that
\begin{enumerate}[\((a)\)]
\item for every \(x \in N^n + B_{\iota}^\nu\), \(\theta(x) = 2\),
\item for every \(x \in \R^\nu \setminus (N^n + B_{2\iota}^\nu)\), \(\theta(x) = 0\).
\end{enumerate}

Given \(0 < \epsilon \le 1\), let \(d_\epsilon : N^n + B_{2\iota}^\nu \to \R\) be defined by
\[
d_\epsilon (x) = \frac{1}{\epsilon} \dist{(\Pi(x), L^{n - \ell - 1})}.
\]

Let \(\lambda : [0, +\infty) \to [0, 1]\) be a Lipschitz continuous function such that
\begin{enumerate}[\((a')\)]
\item for every \(t \le \frac{1}{2}\) and for every \(t \ge 2\), \(\lambda(t) = 0\),
\item \(\lambda(1) = 1\).
\end{enumerate}

Denote by \(\underline\eta : \R^\nu \to N^n\) the function given by Lemma~\ref{lemmaHajlaszContinuous}.
Let \(\eta : \R^\nu \to N^n\) be the map defined by
\[
\eta(x) =
\begin{cases}
\overline f \bigl(\lambda( \theta(x) d_\epsilon(x)), \Pi(x)\bigr)
& \text{if \(x \in N^n + B_{2\iota}^\nu\) and \(\theta(x){d_\epsilon(x)} >  1\),}\\
\underline{\eta} \circ \overline f \bigl(\lambda( \theta(x) d_\epsilon(x)), \Pi(x)\bigr)
& \text{if \(x \in N^n + B_{2\iota}^\nu\) and \(\theta(x){d_\epsilon(x)} \le  1\),}\\
\underline{\eta} \bigl(\overline{\Pi}(x)\bigr)
& \text{if \(x \not\in N^n + B_{2\iota}^\nu\)}.
\end{cases}
\]

\medskip

We first check that \(\eta\) is continuous.
For this purpose we only need to consider the borderline cases:
\begin{enumerate}[\((1)\)]
\item \(x \in N^n + B_{2\iota}^\nu\) and \(\theta(x){d_\epsilon(x)} = 1\),
\item \(x \in \partial(N^n + B_{2\iota}^\nu)\).
\end{enumerate}
In the first case, since \(\lambda(1) = 1\), \(\overline f(1, \cdot) \in T^\ell\) and \(\underline\eta\) is the identity map on \(T^\ell\), we have
\[
\begin{split}
\overline f\bigl(\lambda( \theta(x) d_\epsilon(x)), \Pi(x)\bigr)
& =
\overline f(1, \Pi(x))\\
& =
\underline\eta \bigl( \overline f(1, \Pi(x))\bigr)
=
\underline\eta \circ \overline f\bigl(\lambda( \theta(x) d_\epsilon(x)), \Pi(x)\bigr)
.
\end{split}
\]
In the second case, \(\theta(x) = 0\). 
Since \(\lambda(0) = 0\) and \(\overline{f}(0, \cdot)\) is the identity map on \(N^n\),
\[
\overline f \bigl(\lambda( \theta(x) d_\epsilon(x)), \Pi(x)\bigr)
=
\Pi(x)
=
\overline\Pi(x),
\]
whence
\[
\underline{\eta} \circ \overline f \bigl(\lambda( \theta(x) d_\epsilon(x)), \Pi(x)\bigr)
=
\underline\eta \bigl(\overline\Pi(x)\bigr).
\]

\medskip
We now check that property \((\ref{1648})\) holds.
Indeed, if  \(x \in (N^n + B_\iota^\nu) \setminus \Pi^{-1}(L^{n - \ell - 1} + B_\epsilon^\nu)\), then \(\theta(x) = 2\) and \(d_\epsilon(x) \ge 1\).
Thus,
\[
\lambda(\theta(x) d_\epsilon(x)) = 0.
\]
We then have
\[
\eta(x) 
= 
\overline f (0, \Pi(x))
=
\Pi(x).
\]

\medskip
It remains to establish property \((\ref{1649})\). 
Indeed, if \(x \not\in N^n+ B^{\nu}_{2\iota}\), then \(\eta(x) = \underline{\eta} \bigl(\overline{\Pi}(x)\bigr)\) and the conclusion follows since \(\underline{\eta}\) and \(\overline \Pi(x)\) are both Lipschitz continuous, with Lipschitz constants independent of \(\epsilon\).
If \(x \in N^n+ B^{\nu}_{2\iota}\) and \(\theta(x) d_\epsilon(x) < \frac{1}{2}\), then \(\eta(x) = \underline{\eta} \circ \Pi(x)\) and the estimate follows similarly.
Finally, if \(x \in N^n+ B^{\nu}_{2\iota}\) and \(\theta(x) d_\epsilon(x) \ge \frac{1}{2}\), then
\[
\dist{(\Pi(x), L^{n - \ell - 1})} \ge \frac{\epsilon}{4}.
\] 
By the chain rule and the estimates given by Lemma~\ref{lemmaf},
\[
\abs{D\eta(x)}
\leq \NewConstant \left( \frac{1}{\epsilon} + \frac{1}{\dist{(\Pi(x), L^{n - \ell - 1})}}\right).
\]
Combining both estimates, we get the conclusion.
The proof of the lemma is complete.
\end{proof}

We now have all tools to prove Proposition~\ref{propositionSmoothProjection}.

\begin{proof}[Proof of Proposition~\ref{propositionSmoothProjection}]
Let \(\varphi : \R^\nu \to \R\) be a smooth map  supported in the unit ball \(B_1^\nu\). 
For every \(t > 0\), let \(\varphi_t : \R^\nu \to \R\) be the function defined for \(x \in \R^\nu\) by \(\varphi_t(x) = \frac{1}{t^\nu} \varphi(\frac{x}{t})\). Let \(\iota > 0\) as in the previous lemma.

Given \(0 < \epsilon \le 1\), let \(\zeta : \R^\nu \to [0, 1]\) be a smooth function such that
\begin{enumerate}[\((a)\)]
\item for every \(x \in N^n \setminus (L^{n - \ell - 1} + B_{2 \epsilon}^\nu)\), \(\zeta(x) = 1\),
\label{1645}
\item for every \(x \not\in (N^n + B_\iota^\nu) \setminus \Pi^{-1}(L^{n - \ell - 1} + B_{\epsilon}^\nu)\), \(\zeta(x) = 0\),
\label{1646}
\item for every \(j \in \N_*\),
\[
\norm{D^j \zeta}_{L^\infty(\R^\nu)} \le \frac{\NewConstant}{\epsilon^j},
\]
where \(\SameConstant > 0\) depends on \(j\).
\end{enumerate}

Let \(\eta_\epsilon : \R^\nu \to N^n\) be the Lipschitz continuous map given by the previous lemma and let \(t > 0\) to be chosen below.

By property~\((\ref{1646})\) and by Lemma~\ref{lemmaHajlaszLipschitz}~\((\ref{1648})\), the function
\[
\zeta \eta_\epsilon + (1 - \zeta) \varphi_t * \eta_\epsilon
\]
is smooth in \(\R^\nu\) and  for every \(j \in \N_*\) there exists \(\Constant > 0\) such that
\begin{equation}
\label{eqEstimateHigherOrderDerivatives}
\bignorm{D^j\big(\zeta \eta_\epsilon + (1 - \zeta) \varphi_t * \eta_\epsilon \big)}_{L^\infty(\R^\nu)} \le {\SameConstant} \Big(1 + \frac{1}{t^j} \Big).
\end{equation}
Moreover, by property~\((\ref{1645})\) and by Lemma~\ref{lemmaHajlaszLipschitz}~\((\ref{1648})\),  for every \(x \in N^n \setminus (L^{n - \ell - 1} + B_{2 \epsilon}^\nu)\),
\[
\big(\zeta \eta_\epsilon + (1 - \zeta) \varphi_t * \eta_\epsilon \big)(x) 
=
\eta_\epsilon(x)
=
\Pi(x)
= 
x.
\]
By Lemma~\ref{lemmaHajlaszLipschitz}~\((\ref{1649})\) we have for every \(t > 0\), 
\[
\norm{\varphi_t * \eta_\epsilon - \eta_\epsilon}_{L^\infty(\R^\nu)}
\le 
t \norm{D \eta_\epsilon}_{L^\infty(\R^\nu)}
\le
t \frac{\Constant}{\epsilon}.
\]

Taking
\[
t = \frac{\iota \epsilon}{\SameConstant},
\]
it follows from the previous estimate that the image of \(\zeta \eta_\epsilon + (1 - \zeta) \varphi_t * \eta_\epsilon\) is contained in \(N^n + B_{\iota}^\nu\). 
Hence, the function \(\eta : \R^\nu \to N^n\),
\[
\eta = \Pi \circ \bigl(\zeta \eta_\epsilon + (1 - \zeta) \varphi_t * \eta_\epsilon\bigr),
\]
is well-defined and smooth. 
Property~\((\ref{1655})\) holds with
\[
K = N^n \cap (L^{n - \ell - 1} + B_{2\epsilon}^\nu).
\]
Property~\((\ref{1656})\) also holds since \(K\) is a neighborhood of \(L^{n - \ell - 1}\) in \(N^n\) whose radius is of the order of \(\epsilon\).
By estimate \eqref{eqEstimateHigherOrderDerivatives}, property~\((\ref{1657})\) is also satisfied. 
This completes the proof of the proposition.
\end{proof}


\subsection{Analytic tool}

In this section we establish pointwise estimates of derivatives and fractional derivatives of the map \(\eta \circ u\), where \(\eta\) is a smooth function and \(u\) belongs to \(W^{s, p} \cap L^\infty\).
In the case where \(s\) is an integer, this estimate follows from the classical chain rule for higher order derivatives:

\begin{proposition}
\label{propositionPointwiseEstimate}
Let \(k \in \N_*\).
If \(u \in W^{k, p}(Q^m; \R^\nu) \cap W^{1, kp}(Q^m; \R^\nu)\), then
for every smooth map \(\eta : \R^\nu \to \R^\nu\) and
for every \(j \in \{1, \dots, k\},\) there exists a measurable function \(G_j : \Omega \to \R\) such that
\[
\abs{D^j(\eta \circ u)} \le [\eta]_{C^j(\R^\nu)} G_j
\]
and
\[
\norm{G_j}_{L^p(Q^m)} \le C,
\]
for some constant \(C > 0\) depending on \(k\), \(p\), \(m\), \(\norm{u}_{W^{k, p}(Q^m)}\) and \(\norm{u}_{W^{1, kp}(Q^m)}\).
\end{proposition}

We use the following notation:
\[
[\eta]_{C^j(\R^\nu)} = \sum_{i=1}^j \norm{D^i\eta}_{L^\infty(\R^\nu)}.
\]

\begin{proof}
We first observe that \(\eta \circ u \in W^{k, p}(Q^m; \R^\nu)\).
By the chain rule, 
\[
\begin{split}
|D^{j}(\eta \circ u)(x)|
& \leq \NewConstant \sum_{i=1}^j \abs{D^i \eta(u(x))} \sum_{\substack{1\leq t_1\leq \dots \leq t_i\leq j,\\ t_1+\dots+t_i = j}}|D^{t_1}u(x)| \dotsm |D^{t_i}u(x)|\\
& \leq \SameConstant [\eta]_{C^j(\R^\nu)} \sum_{i=1}^j \sum_{\substack{1\leq t_1\leq \dots \leq t_i\leq j,\\ t_1+\dots+t_i = j}}|D^{t_1}u(x)|\dotsm |D^{t_i}u(x)|.
\end{split}
\] 
Let
\[
G_j
= \SameConstant \sum_{i=1}^j \sum_{\substack{1\leq t_1\leq \dots \leq t_i\leq j,\\ t_1+\dots+t_i = j}}|D^{t_1}u|\dotsm |D^{t_i}u|
\]
Since the map \(u\) in the statement belongs to \(W^{k,p}(Q^m ; \R^\nu) \cap W^{1,kp}(Q^m ; \R^\nu)\), it follows from the Gagliardo-Nirenberg interpolation inequality \cites{Gagliardo-1959, Nirenberg-1959} that
\[
D^i u \in L^{\frac{jp}{i}}(Q^m).
\]
By H\"older's inequality, we deduce that \(G_j \in L^p(Q^m)\).
\end{proof}

We now establish a counterpart of the previous proposition for the fractional derivative introduced by Maz\cprime ya and Shaposhnikova~\cite{Mazya-Shaposhnikova}.
More precisely, given \(0 < \sigma < 1\), \(1 \le p < +\infty\), a domain \(\Omega \subset \R^m\) and a measurable function \(u : \Omega \to \R^\nu\), define for \(x \in \Omega\),
\[
\displaystyle D^{\sigma, p}u(x) = \bigg(\int\limits_{\Omega} \frac{\abs{u(x) - u(y)}^p}{\abs{x - y}^{m + \sigma p}} \dif y\bigg)^{1/p}.
\]
We extend this definition for any \(s>0\) such that \(s\notin \N\) as follows: 
\[
D^{s, p}u =  D^{\sigma, p} (D^k u),
\]
where \(k = \floor{s}\) is the integral part of \(s\) and \(\sigma = s - \floor{s}\) is the fractional part of \(s\). 
Using this notation, we have
\[
[D^k u]_{W^{\sigma, p}(\Omega)} = \norm{D^{s, p} u}_{L^p(\Omega)}.
\]

\begin{proposition}
\label{propositionPointwiseEstimateWsp}
Let \(s > 1\) be such that \(s \not\in \N\).
If \(u \in W^{s, p}(Q^m; \R^\nu) \cap W^{1, sp}(Q^m; \R^\nu)\), then for every smooth map \(\eta : \R^\nu \to \R^\nu\) there exists a measurable function \(H : Q^m \to \R\) such that
\[
\abs{D^{s, p}(\eta \circ u)} \le  [\eta]_{C^{k+1}(\R^\nu)}^{\sigma} [\eta]_{C^k(\R^\nu)}^{1 - \sigma} H
\]
and
\[
\norm{H}_{L^p(Q^m)} \le C,
\]
for some constant \(C > 0\) depending on \(s\), \(p\), \(m\), \(\norm{u}_{W^{s, p}(Q^m)}\) and \(\norm{u}_{W^{1, sp}(Q^m)}\).
\end{proposition}

This proposition implies a theorem of Brezis and Mironescu~\cite{Brezis-Mironescu}*{Theorem~1.1} concerning the boundedness of the composition operator from \(W^{s, p} \cap W^{1, sp}\) into \(W^{s, p}\). 
A more elementary proof of the same result has been provided by Maz\cprime ya and Shaposhnikova~\cite{Mazya-Shaposhnikova};
our proof of Proposition~\ref{propositionPointwiseEstimateWsp} is based on their strategy.

We begin with the following pointwise estimate of Maz\cprime ya and Shaposhnikova~\cite{Mazya-Shaposhnikova}*{Lemma}:

\begin{lemma}\label{Bou}
Let \(q \ge 1\). 
If \(v \in W^{1, q}_{\mathrm{loc}} (\R^m; \R^\nu)\), then for \(x \in \R^m\),
\[
\bigl(D^{\sigma, q} v(x)\bigr)^q
\le
C \bigl(\mathcal{M}\abs{Dv}^q(x)\bigr)^{\sigma} \bigl(\mathcal{M}|v|^q(x)\bigr)^{1-\sigma},
\]
for some constant \(C > 0\) depending on \(m\) and \(q\).
\end{lemma}

The maximal function associated to a nonnegative function $f\in L^1_\mathrm{loc}(\R^m)$ is defined for \(x \in \R^m\) by
\[
\mathcal{M}f(x)=\sup_{\rho >0} \frac{1}{|B_\rho^m|} \int\limits_{B_\rho^m (x)} f.
\]
For completeness, we prove Lemma~\ref{Bou} using a property of the maximal function due to Hedberg~\cite{Hedberg}*{Lemma}:

\begin{lemma}
\label{lemmaHedberg}
Let \(f \in L^1_{\mathrm{loc}} (\R^m)\) be a nonnegative function and let  \(\delta >0\).
For every \(x \in \R^m\) and \(\rho> 0\), 
\begin{align*}
\int\limits_{B_\rho^m (x)} \frac{f (y)}{\abs{y - x}^{m - \delta}} \dif y
 & \le C \rho^{\delta} \mathcal{M} f (x),\\
 \int\limits_{\R^m \setminus B_\rho^m (x)} \frac{f (y)}{\abs{y - x}^{m + \delta}} \dif y
 & \le \frac{C}{\rho^{\delta}} \mathcal{M} f (x),
\end{align*}
for some constant \(C > 0\) depending on \(m\) and \(\delta\).
\end{lemma}
\begin{proof}
We briefly sketch the proof of Hedberg for the first estimate. The proof of the second one is similar. One has
\[
\begin{split}
\int\limits_{B_\rho^m (x)} \frac{f (y)}{\abs{y - x}^{m - \delta}} \dif y
& = \sum_{i = 0}^\infty \int\limits _{B^m_{ \rho 2^{-i}} (x) \setminus B^m_{\rho 2^{-i-1} } (x)} \frac{f (y)}{\abs{x - y}^{m - \delta}} \dif y \\
& \le \NewConstant \rho^\delta \sum_{i = 0}^\infty  2^{-\delta i} \mathcal{M} f (x)
\le \Constant \rho^\delta \mathcal{M} f (x).\qedhere
\end{split}
\]
\end{proof}

\begin{proof}[Proof of Lemma~\ref{Bou}]
Let \(\rho > 0\).
By Hardy's inequality \cite{Mazya}*{Section~1.3},
\[
\int\limits_{B_\rho^m(x)} \frac{\abs{v(x) - v(y)}^q}{\abs{x - y}^{m + \sigma q}} \dif y
\le \NewConstant \int\limits_{B_\rho^m(x)} \frac{\abs{Dv(y)}^q}{\abs{x - y}^{m - (1 - \sigma) q}} \dif y.
\]
Thus, by Hedberg's lemma,
\[
 \int\limits_{B_\rho^m(x)} \frac{\abs{v(x) - v(y)}^q}{\abs{x - y}^{m + \sigma q}} \dif y
 \le 
\Constant \rho^{(1 - \sigma)q} \cM \abs{Dv}^q (x).
\]
Since 
\[
\abs{v(x) - v(y)}^q 
\le \Constant \big( \abs{v(x)}^q + \abs{v(y)}^q \big),
\]
by an explicit integral computation and by Hedberg's lemma,
\[
\int\limits_{\R^m \setminus B_\rho^m(x)} \frac{\abs{v(x) - v(y)}^q}{\abs{x - y}^{m + \sigma q}} \dif y
\le \frac{\Constant}{\rho^{\sigma q}} \big( \abs{v(x)}^q + \cM\abs{v}^q(x) \big)
\le
\frac{\Constant}{\rho^{\sigma q}} \cM\abs{v}^q(x).
\]
We conclude that
\[
\bigl(D^{\sigma, q} v(x)\bigr)^q 
\le C_2 \rho^{(1 - \sigma)q} \cM \abs{Dv}^q(x) + \frac{\SameConstant}{\rho^{\sigma q}} \cM\abs{v}^q(x).
\]
Minimizing the right-hand side with respect to \(\rho\), we deduce the pointwise estimate.
\end{proof}

The following lemma is implicitly proved in 
\cite{Mazya-Shaposhnikova}*{Section~2}:

\begin{lemma}\label{MaSh}
Let \(0 < \sigma < 1\), \(1 \le p < +\infty\) and \(i \in \N_*\). 
If for every \(\alpha \in \{1, \dots, i\}\), 
\(v_\alpha \in L^{q_\alpha}(\R^m)\) and \(Dv_\alpha \in L^{r_\alpha}(\R^m)\), where \(1 < r_\alpha < q_\alpha\) and
\[
  \frac{1 - \sigma}{q_\alpha} + \frac{\sigma}{r_\alpha} + \sum_{\substack{\beta = 1\\\beta \ne \alpha}}^i \frac{1}{q_\beta} = \frac{1}{p},
\]
then \(\prod\limits_{\alpha =1}^i v_\alpha \in W^{\sigma, p}(\R^m)\) and
\[
\textstyle 
\bigl[\prod\limits_{\alpha=1}^i v_\alpha \bigr]_{W^{\sigma, p}(\R^m)}
\displaystyle 
\leq C \sum_{\alpha = 1}^i \bigg(\|v_\alpha \|_{L^{q_\alpha}(\R^m)}^{1-\sigma} \|D v_\alpha\|_{L^{r_\alpha}(\R^m)}^{\sigma} 
\prod_{\substack{\beta = 1\\\beta \ne \alpha}}^i \|v_\beta \|_{L^{q_\beta}(\R^m)} \bigg),
\]
for some constant \(C > 0\) depending on \(m\), \(\sigma\),  \(r_1\), \dots, \(r_i\), \(q_1\), \dots, \(q_i\).
\end{lemma}

\begin{proof}
We first consider the case of dimension \(m=1\).
Note that
\begin{equation}
\label{eqProductTriangleInequality}
\textstyle \bigabs{\prod\limits_{\alpha=1}^i v_\alpha(x) - \prod\limits_{\alpha=1}^i v_\alpha(y)}
\displaystyle 
\le \sum_{\alpha = 1}^i \bigabs{v_1(x) \dotsm v_{\alpha - 1}(x) \bigl(v_\alpha(x) - v_\alpha(y)\bigr) v_{\alpha + 1}(y) \dotsm v_i(y)}.
\end{equation}
Thus, the left-hand side is bounded from above by a sum of functions of the form
\[
\underline{f}_\alpha (x)\, \abs{v_\alpha(x) - v_\alpha(y)}\, \overline{f}_\alpha(y).
\]

By the Fundamental theorem of Calculus, for every \(x, y \in \R\),
\[
\abs{v_\alpha(x) - v_\alpha(y)} 
\le 2 \abs{x - y} \cM{\abs{v_\alpha'}}(x).
\]
Thus, for every \(\rho > 0\),
\[
\int\limits_{B_\rho^1(x)} \frac{\abs{v_\alpha(x) - v_\alpha(y)}^p}{\abs{x - y}^{1 + \sigma p}} (\overline{f}_\alpha(y))^p \dif y 
\le
\NewConstant \big(\cM\abs{v_\alpha'}(x)\big)^p \int\limits_{B_\rho^1(x)} \frac{(\overline{f}_\alpha(y))^p}{\abs{x - y}^{1 - (1 -\sigma) p}}  \dif y. 
\]
By Hedberg's lemma, we get
\[
\int\limits_{B_\rho^1(x)} \frac{\abs{v_\alpha(x) - v_\alpha(y)}^p}{\abs{x - y}^{1 + \sigma p}} (\overline{f}_\alpha(y))^p \dif y 
\le 
\Constant\rho^{(1 -\sigma) p} \big(\cM\abs{v_\alpha'}(x)\big)^p \cM  (\overline{f}_\alpha)^p(x).
\]
Next, we write
\[
\abs{v_\alpha(x) - v_\alpha(y)}^p (\overline{f}_\alpha(y))^p
\le \Constant \big( \abs{v_\alpha(x)}^p (\overline{f}_\alpha(y))^p + \abs{v_\alpha(y)}^p (\overline{f}_\alpha(y))^p \big).
\]
By Hedberg's lemma, we also have
\[
\int\limits_{\R \setminus B_\rho^1(x)} \frac{\abs{v_\alpha(x) - v_\alpha(y)}^p}{\abs{x - y}^{1 + \sigma p}} (\overline{f}_\alpha(y))^p \dif y 
\le 
\frac{\Constant}{\rho^{\sigma p}} \bigl(\abs{v_\alpha(x)}^p \cM(\overline{f}_{\alpha})^p(x) + \cM\abs{v_\alpha \overline{f}_{\alpha}}^p(x) \bigr).
\]
We conclude that
\begin{multline*}
\int\limits_{\R} \frac{\abs{v_\alpha(x) - v_\alpha(y)}^p}{\abs{x - y}^{1 + \sigma p}} (\overline{f}_\alpha(y))^p \dif y 
\\
\le 
C_2 \rho^{(1 -\sigma) p} \bigl(\cM\abs{v_\alpha'}(x)\bigr)^p \cM(\overline{f}_\alpha)^p (x) \\
+ \frac{\SameConstant}{\rho^{\sigma p}} \bigl(\abs{v_\alpha(x)}^p \cM (\overline{f}_\alpha)^p (x) + \cM\abs{v_\alpha \overline{f}_\alpha}^p(x) \bigr).
\end{multline*}
Minimizing the right hand side with respect to $\rho$, we then get
\begin{multline*}
\int\limits_{\R} \frac{\abs{v_\alpha(x) - v_\alpha(y)}^p}{\abs{x - y}^{1 + \sigma p}} (\overline{f}_\alpha(y))^p \dif y 
\\
\le
\Constant \bigl(\mathcal{M}\abs{v_\alpha'}(x)\bigr)^{\sigma p} (\cM (\overline{f}_\alpha)^p (x))^\sigma \bigl(\abs{v_\alpha(x)}^p \cM (\overline{f}_\alpha)^p (x) + \cM \abs{v_\alpha \overline{f}_\alpha}^p (x) \bigr)^{1-\sigma}.
\end{multline*}
Thus,
\begin{multline}
\label{eqProductEstimate}
\int\limits_{\R} \int\limits_{\R} (\underline{f}_\alpha(x))^p \frac{\abs{v_\alpha(x) - v_\alpha(y)}^p}{\abs{x - y}^{1 + \sigma p}} (\overline{f}_\alpha(y))^p \dif x \dif y \\
\le
{\SameConstant} \int\limits_{\R}  (\underline{f}_\alpha)^p \bigl(\mathcal{M}\abs{v_\alpha'}\bigr)^{\sigma p} \bigl(\cM (\overline{f}_\alpha)^p \bigr)^\sigma \bigl(\abs{v_\alpha}^p \cM (\overline{f}_\alpha)^p + \cM \abs{v_\alpha \overline{f}_\alpha}^p\bigr)^{1-\sigma} .
\end{multline}
Let \(\frac{1}{\underline q_\alpha} = \sum\limits_{\beta =1}^{\alpha-1} \frac{1}{q_\beta}\) and \(\frac{1}{\overline q_\alpha} = \sum\limits_{\beta = \alpha + 1}^i \frac{1}{q_\beta}\), so that by assumption,
\[
\frac{1}{\underline{q}_\alpha} + \frac{\sigma}{r_\alpha} + \frac{1-\sigma}{q_\alpha} + \frac{1}{\overline{q}_\alpha}
= \frac{1}{p}.
\]
By H\"older's inequality,
\begin{multline*}
\int\limits_{\R}  (\underline{f}_\alpha)^p \bigl(\mathcal{M}\abs{v_\alpha'}\bigr)^{\sigma p} \abs{v_\alpha}^{(1-\sigma)p} \cM (\overline{f}_\alpha)^p 
\\
\le 
\norm{\underline{f}_\alpha}_{L^{\underline{q}_\alpha}(\R)}^p \norm{\mathcal{M}\abs{v_\alpha'}}_{L^{r_\alpha}(\R)}^{\sigma p} \norm{v_\alpha}_{L^{q_\alpha}(\R)}^{(1-\sigma)p} \norm{(\mathcal{M} (\overline{f}_\alpha)^p )^{1/p}}_{L^{\overline{q}_\alpha}(\R)}^p.
\end{multline*}

We estimate the right hand side as follows. 
By H\"older's inequality,
\[
\norm{\underline{f}_\alpha}_{L^{\underline{q}_\alpha}(\R)}
\le \prod_{\beta = 1}^{\alpha - 1} {\|v_\beta \|_{L^{q_\beta}(\R)}}.
\]
Since \(r_\alpha > 1\), by the Maximal theorem~\cite{Stein}*{Chapter~1, Theorem~1},
\[
\norm{\mathcal{M}\abs{v_\alpha'}}_{L^{r_\alpha}(\R)} 
\le \Constant 
\norm{v_\alpha'}_{L^{r_\alpha}(\R)}.
\]
Since \(\overline q_\alpha/p > 1\), by the Maximal theorem and by H\"older's inequality,
\[
\begin{split}
\norm{(\mathcal{M} (\overline{f}_\alpha)^p)^{1/p}}_{L^{\overline{q}_\alpha}(\R)} 
& =
\norm{\mathcal{M} (\overline{f}_\alpha)^p}_{L^{\overline{q}_\alpha/p}(\R)}^{1/p} 
\\
& \le \Constant 
\norm{(\overline{f}_\alpha)^p}_{L^{\overline{q}_\alpha/p}(\R)}^{1/p} \\
& = 
\SameConstant \norm{\overline{f}_\alpha}_{L^{\overline{q}_\alpha}(\R)} 
\le 
\prod_{\beta = \alpha + 1}^i {\|v_{\beta} \|_{L^{q_{\beta}}(\R)} }.
\end{split}
\]
Combining these estimates we get
\[
\int\limits_{\R}  (\underline{f}_\alpha)^p \bigl(\mathcal{M}\abs{v_\alpha'}\bigr)^{\sigma p} \abs{v_\alpha}^{(1-\sigma)p} \cM (\overline{f}_\alpha)^p 
\le 
\|v_\alpha \|_{L^{q_\alpha}(\R)}^{(1-\sigma)p} \|v_\alpha'\|_{L^{r_\alpha}(\R)}^{\sigma p} 
\prod_{\substack{\beta = 1\\\beta \ne \alpha}}^i \|v_\beta \|_{L^{q_\beta}(\R)}^p.
\]
Similarly,
\[
\int\limits_{\R}  (\underline{f}_\alpha)^p \bigl(\mathcal{M}\abs{v_\alpha'}\bigr)^{\sigma p} \bigl(\cM (\overline{f}_\alpha)^p \bigr)^\sigma \bigl(\cM \abs{v_\alpha \overline{f}_\alpha}^p\bigr)^{1-\sigma} 
\le
\|v_\alpha \|_{L^{q_\alpha}(\R)}^{(1-\sigma)p} \|v_\alpha'\|_{L^{r_\alpha}(\R)}^{\sigma p} 
\prod_{\substack{\beta = 1\\\beta \ne \alpha}}^i \|v_\beta \|_{L^{q_\beta}(\R)}^p.
\]
Therefore, by \eqref{eqProductEstimate},
\begin{multline*}
\int\limits_{\R} \int\limits_{\R} (\underline{f}_\alpha(x))^p \frac{\abs{v_\alpha(x) - v_\alpha(y)}^p}{\abs{x - y}^{1 + \sigma p}} (\overline{f}_\alpha(y))^p \dif x \dif y
\\
\le
\|v_\alpha \|_{L^{q_\alpha}(\R)}^{(1-\sigma)p} \|v_\alpha'\|_{L^{r_\alpha}(\R)}^{\sigma p} 
\prod_{\substack{\beta = 1\\\beta \ne \alpha}}^i \|v_\beta \|_{L^{q_\beta}(\R)}^p.
\end{multline*}
In view of the triangle inequality~\eqref{eqProductTriangleInequality}, we have the conclusion in dimension \(m = 1\).

\medskip

When \(m > 1\), we reduce the problem to the one dimensional case using the estimate \cite{Adams_1975}*{Lemma~7.44}
\[
\textstyle
\bigl[ \prod\limits_{\alpha=1}^i v_\alpha \bigr]_{W^{\sigma, p}(\R^m)}
\leq 
\displaystyle 
\NewConstant  \sum_{j=1}^m 
\bigg( 
\int\limits_{\R^m} \int\limits_\R 
\frac{\bigabs{\prod\limits_{\alpha=1}^i v_{\alpha}(x+te_j)
-\prod\limits_{\alpha=1}^i v_{\alpha}(x)}^p }
{t^{1+\sigma p}}
\dif t \dif x
\bigg)^{1/p},
\]
where \((e_1, \dots, e_m)\) is the canonical basis of \(\R^m\).

We only estimate the first term of the sum in the right hand side. 
We write any \(x \in \R^m\) as \(x = (x_1, x') \in \R \times \R^{m - 1}\).
For  \(x' \in \R^{m - 1}\), we apply the case \(m=1\) to the function \(x_1 \in \R \mapsto v_\alpha(x_1, x')\). 
This gives
\begin{multline*}
\int\limits_\R  \int\limits_\R 
\frac{\bigabs{\prod\limits_{\alpha=1}^i v_\alpha(x_1+t, x')-\prod\limits_{\alpha=1}^i v_{\alpha}(x_1, x')}^p}{t^{1+\sigma p}} \dif t \dif x_1 
\\
\leq \NewConstant 
\sum_{\alpha = 1}^i \bigg(\|v_\alpha (\cdot, x') \|_{L^{q_\alpha}(\R)}^{(1-\sigma)p} \|\partial_1 v_\alpha(\cdot, x')\|_{L^{ r_\alpha}(\R)}^{\sigma p} 
\prod_{\substack{\beta = 1\\\beta \ne \alpha}}^i \|v_\beta(\cdot, x') \|_{L^{q_\beta}(\R)}^p \bigg).
\end{multline*}
Integrating both sides with respect to \(x'\) over \(\R^{m - 1}\), we obtain by Fubini's theorem,
\begin{multline*}
\int\limits_{\R^m} \int\limits_\R  
\frac{\bigabs{\prod\limits_{\alpha=1}^i v_\alpha(x+te_1)-\prod\limits_{\alpha=1}^i v_{\alpha}(x)}^p}{t^{1+\sigma p}} \dif t \dif x 
\\
\leq \Constant  
\sum_{\alpha = 1}^i 
\int\limits_{\R^{m-1}} \|v_\alpha(\cdot, x') \|_{L^{q_\alpha}(\R)}^{(1-\sigma)p} \|\partial_1 v_\alpha(\cdot, x')\|_{L^{r_\alpha}(\R)}^{\sigma p} 
\prod_{\substack{\beta = 1\\\beta \ne \alpha}}^i \|v_\beta(\cdot, x') \|^p_{L^{q_\beta}(\R)}\dif x'.
\end{multline*}
Using H\"older's inequality with exponents \(\frac{q_{\alpha}}{(1-\sigma) p}\), \(\frac{r_\alpha}{\sigma p}\) and \(\frac{q_\beta}{p}\) for \(\beta \ne \alpha\), we get the desired result.
\end{proof}

For \(p > 1\), there is an alternative proof of Lemma~\ref{MaSh} using the Triebel-Lizorkin spaces \(F^{\sigma}_{t, p} (\R^m)\), based on the imbedding of the product of functions in such spaces.
By the Gagliardo-Nirenberg interpolation inequality \citelist{\cite{Oru}\cite{Brezis-Mironescu-2001}*{Lemma~3.1}},
\[
\norm{v_\alpha}_{F^{\sigma}_{s_\alpha, p} }
\le C \norm{v_\alpha}_{L^{q_\alpha}}^{1-\sigma} \norm{v_\alpha}_{W^{1, r_\alpha}}^{\sigma},
\]
with
\[
 \frac{1}{s_\alpha} = \frac{1 - \sigma}{q_\alpha} + \frac{\sigma}{r_\alpha}.
\]
Since for every \(\alpha \in \{1, \dotsc, i\}\),
\[
  \frac{1}{s_\alpha} + \sum_{\substack{\beta = 1\\\beta \ne \alpha}}^i \frac{1}{q_\beta} =
\frac{1 - \sigma}{q_\alpha} + \frac{\sigma}{r_\alpha} + \sum_{\substack{\beta = 1\\\beta \ne \alpha}}^i \frac{1}{q_\beta} = \frac{1}{p},
\]
if \(p>1\), then it follows that \cite{RunstSickel}*{p.~345}
\[
\textstyle \prod\limits_{\alpha = 1}^i v_\alpha \in F^\sigma_{p, p} (\R^m) = W^{\sigma, p} (\R^m).
\]

\begin{proof}[Proof of Proposition~\ref{propositionPointwiseEstimateWsp}]
By continuous extension of functions in Sobolev spaces to the whole space, it suffices to establish the estimate on \(\R^m\) instead of \(Q^m\).
By the chain rule and by the triangle inequality, we have for $x, y \in \R^m$,
\begin{multline*}
|D^{k}(\eta \circ u)(x)-D^{k} (\eta \circ u)(y)|
\\
\leq \NewConstant \sum_{i=1}^{k}
\sum_{\substack{1\leq t_1\leq\dots\leq t_i\leq k,\\ t_1+\dots+t_i= k}}
\big|D^i\eta(u(x))[D^{t_1} u(x),\dots,D^{t_i} u(x)]\\
-D^{i}\eta(u(y))[D^{t_1} u(y),\dots,D^{t_i} u(y)] \big|.
\end{multline*}

Given \(1\leq t_1\leq \dots \leq t_i\leq k\) such that \(t_1+\dots +t_i = k\),
by the triangle inequality we have
\begin{multline*}
\bigl|D^i\eta(u(x))[D^{t_1} u(x), \dots, D^{t_i} u(x)]
- D^{i}\eta(u(y))[D^{t_1} u(y),\dots,D^{t_i} u(y)] \bigr| \\
\leq F_{t_1, \dots, t_i}(x, y) + G_{t_1, \dots, t_i}(x, y) 
\end{multline*}
with
\[
F_{t_1, \dots, t_i}(x, y) 
= \bigl|D^i\eta(u(x)) - D^i\eta(u(y)) \bigr||D^{t_1} u(x)|\dotsm |D^{t_i} u(x)|
\]
and
\[
G_{t_1, \dots, t_i}(x, y)
= |D^i\eta(u(y))| \big|D^{t_1} u(x) \otimes \dots \otimes D^{t_i}u(x)
-D^{t_1} u(y) \otimes \dots \otimes D^{t_i}u(y)\big|.
\]
The notation \(\otimes\) is used in the following sense: 
if \(f_\alpha = (f_{\alpha}^{1}, \dots, f_{\alpha}^{\nu}) : (\R^m)^{t_\alpha}\to \R^\nu\) is a \(t_\alpha\)-linear transformation for \(\alpha \in \{1, \dots, i\}\), then \(f_1\otimes \dots \otimes f_i\) 
is the \((\sum\limits_{\alpha=1}^i t_{\alpha})\)-linear transformation 
\[
(X_1,\dots, X_i)\in \prod_{\alpha=1}^i (\R^m)^{t_\alpha} \mapsto  \big(f_{1}^{j_1}(X_1) \dots f_{i}^{j_i}(X_i) \big)_{1\leq j_1,\dots, j_i \leq \nu} \in  \R^{\nu^i}.
\]
Thus,
\begin{multline*}
D^{s, p}(\eta \circ u)(x)\\
\le
\SameConstant \sum_{i=1}^{k}
\sum_{\substack{1\leq t_1\leq\dots\leq t_i\leq k\\ t_1+\dots+t_i= k}} 
\biggl( \int\limits_{\R^m}\frac{F_{t_1, \dots, t_i}(x, y)^p}{|x-y|^{m+\sigma p}} \dif y \biggr)^{1/p}
+ 
\biggl( \int\limits_{\R^m}\frac{G_{t_1, \dots, t_i}(x, y)^p}{|x-y|^{m+\sigma p}} \dif y
\biggr)^{1/p}.
\end{multline*}

We have
\[
\int\limits_{\R^m}\frac{F_{t_1, \dots, t_i}(x, y)^p}{|x-y|^{m+\sigma p}} \dif y\\
= \big(D^{\sigma, p}(D^i\eta \circ u)(x) \big)^p
|D^{t_1} u(x)|^p \dotsm |D^{t_i} u(x)|^p.
\]
By Lemma~\ref{Bou},
\[
\big(D^{\sigma, p}(D^i\eta \circ u)(x) \big)^p 
\le \Constant \big(\cM\abs{D(D^i\eta \circ u)}^p(x) \big)^\sigma \big(\cM\abs{D^i\eta \circ u}^p(x) \big)^{1 - \sigma}. 
\]
Moreover, for every \(i \in \{1, \dots, k\}\),
\[
\abs{D(D^i\eta \circ u)} 
\le [\eta]_{C^{k+1}(\R^\nu)}\abs{Du} 
\quad \text{and} \quad
\abs{D^i\eta \circ u}
\le [\eta]_{C^{k}(\R^\nu)}.
\]
Hence, 
\begin{multline*}
\biggl(
\int\limits_{\R^m}\frac{F_{t_1, \dots, t_i}(x, y)^p}{|x-y|^{m+\sigma p}} \dif y
\biggr)^{1/p}\\ 
\leq \SameConstant  [\eta]_{C^{k+1}(\R^\nu)}^{\sigma} [\eta]_{C^k(\R^\nu)}^{1 - \sigma} \big(\cM\abs{Du}^p(x) \big)^\frac{\sigma}{p} |D^{t_1} u(x)| \dotsm |D^{t_i} u(x)|.
\end{multline*}

 
Since \(D u \in L^{sp}(\R^m)\) and \(s > 1\), by the Maximal Theorem we have
\[
 \mathcal{M}\abs{D u}^p \in L^s (\R^m).
\]
By Hölder's inequality it follows that
\[
\big(\mathcal{M}\abs{Du}^{p} \big)^\frac{\sigma}{p} |D^{t_1} u| \dotsm |D^{t_i} u| \in L^p(\R^m).
\]

Next,
\[
\biggl(
\int\limits_{\R^m}\frac{G_{t_1, \dots, t_i}(x, y)^p}{|x-y|^{m+\sigma p}} \dif y 
\biggr)^{1/p}
\le [\eta]_{C^k(\R^\nu)} D^{\sigma, p}(D^{t_1} u \otimes \dots \otimes D^{t_i}u)(x).
\]

If \(t_i = k\), then \(i=1\) and this estimate becomes
\[
\bigg(\int\limits_{\R^m}\frac{G_{k}(x, y)^p}{|x-y|^{m+\sigma p}} \dif y \bigg)^{1/p} 
\le [\eta]_{C^k(\R^\nu)} D^{s, p} u (x).
\]
By assumption on \(u\), the right-hand side belongs to \(L^p(\R^m)\).

If \(t_i <k\), then each component of \(D^{t_1}u \otimes \dots \otimes D^{t_i}u\) is the product of \(i\) functions \(v_{t_1}, \dots, v_{t_i}\) with \(v_{t_\alpha} \in L^{\frac{sp}{t_\alpha}}(\R^m)\) and \(Dv_{t_\alpha} \in L^{\frac{sp}{{t_\alpha}+1}}(\R^m)\).
Then, by Lemma~\ref{MaSh}, we get
\[
D^{\sigma, p}(D^{t_1} u \otimes \dots \otimes D^{t_i}u) \in L^p(\R^m).
\]
The proof of the proposition is complete.
\end{proof}



\section{Strong density}

We rely on an averaging argument due to Federer and Fleming~\cites{Federer-Fleming, Hardt-Kinderlehrer-Lin} based on the following observation:

\begin{lemma}
\label{lemmefubini}
Let \(u : Q^m \to \R^\nu\) be a measurable function. 
For every measurable function \(f : Q^m \to \R\) and for every Borel measurable set \(E \subset \R^\nu\),
\[
\int\limits_{\R^\nu}  \int\limits_{u^{-1}(E + \xi)} \abs{f(x)} \dif x  \dif \xi
= \cH^{\nu}(E) \norm{f}_{L^1(Q^m)}. 
\]
\end{lemma}

We shall apply this lemma with \(E = \Pi^{-1}(K)\) where \(K \subset N^n\) is a compact set and \(\Pi : N^n + \overline B_\iota^\nu \to N^n\) is the nearest point projection. 
In this case, by the coarea formula we have
\[
\cH^{\nu}(E) \le C \cH^{n}(K).
\]

\begin{proof}
We may assume that \(f\) is a nonnegative function.
For every \(\xi \in \R^\nu\),
\[
\int\limits_{u^{-1}(E + \xi)} f(x) \dif x = \int\limits_{Q^m} f(x) \chi_{E}(u(x) - \xi) \dif x.
\]
By Fubini's theorem,
\[
\int\limits_{\R^\nu} \int\limits_{u^{-1}(E + \xi)} f(x) \dif x  \dif \xi
= \int\limits_{Q^m} f(x)  \biggl( \int\limits_{\R^\nu} \chi_{E}(u(x) - \xi)  \dif \xi \biggr) \dif x.
\]
Using the change of variable \(z = u(x) - \xi\) with respect to \(\xi\), we get
\[
\begin{split}
\int\limits_{\R^\nu} \int\limits_{u^{-1}(E + \xi)} f(x) \dif x  \dif \xi
& = \int\limits_{Q^m} f(x)  \biggl( \int\limits_{\R^\nu} \chi_{E}(z)  \dif z \biggr) \dif x\\
& = \int\limits_{Q^m} f(x)  \cH^\nu(E)  \dif x
=  \cH^\nu(E) \int\limits_{Q^m} f(x) \dif x.
\end{split}
\]
This gives the conclusion.
\end{proof}

\begin{proof}[{Proof of Theorem~\ref{deuxiemetheorem}}]
Given \(u \in W^{s,p}(Q^{m}; N^n)\), the restriction to \(Q^{m}\)  of the maps \(u_{\gamma} \in W^{s,p}(Q^{m}_{1+\gamma}; N^n)\) defined for \(x \in Q^{m}_{1+\gamma}\) by \(u_{\gamma}(x) = u (x/(1+\gamma))\) converge strongly to \(u\) in \(W^{s,p}(Q^{m}; N^n)\) when \(\gamma\) tends to zero.
We may thus assume that \(u \in W^{s, p}(Q^{m}_{1+\gamma}; N^n)\).

Let \(\varphi : \R^m \to \R\) be a smooth mollifier such that \(\supp{\varphi} \subset Q^m\). 
For every \(0 < t \le \gamma\), the convolution \(\varphi_t * u\) is well-defined and converges to \(u\) in \(W^{s, p}(Q^m; \R^\nu)\) as \(t\) tends to zero.

The nearest point projection \(\Pi\) onto \(N^n\) is well-defined and smooth on  \(N^n + \overline B^\nu_{\iota}\)  for some \(\iota >0\).
Let \(\overline\Pi : \R^\nu \to \R^\nu\) be a smooth extension of the projection \(\Pi\) to \(\R^\nu\).
The image of this map \(\overline{\Pi}\) need not be contained in \(N^n\).

For every \(\xi \in B^{\nu}_\iota\), we consider the map \(P_\xi : \R^\nu \to \R^\nu\) defined for every \(x \in \R^\nu\) by 
\[
P_\xi (x)= \overline\Pi(x - \xi).
\]
There exists \(0 < \delta \le \iota \) such that for every \(\xi \in B^{\nu}_\delta\), the map \(P_\xi |_{N^n} : N^n \to N^n\) is a smooth diffeomorphism.
Given a smooth map \(\eta : \R^\nu \to N^n\) and \(\xi \in B^{\nu}_{\delta}\), let 
\[
\eta_{\xi} = (P_{\xi}|_{N^n})^{-1} \circ \eta \circ P_{\xi}.
\]

Our goal is to approximate \(u\) by a family of maps of the form
\[
\eta_{\xi} \circ (\varphi_{t} * u),
\]
for some \(\xi \in B^{\nu}_{\delta}\) and \(0 < t \le \gamma\).
By the triangle inequality,
\begin{multline}
\label{eqStrongDensityTriangleInequality}
\norm{\eta_{\xi} \circ (\varphi_{t} * u) - u}_{W^{s,p}(Q^{m})}\\
\leq \norm{\eta_{\xi} \circ (\varphi_{t} * u) - \eta_{\xi} \circ u}_{W^{s,p}(Q^m)} + \norm{\eta_{\xi}\circ u - u}_{W^{s,p}(Q^m)}.
\end{multline}
Since \(\eta_{\xi}\) is a smooth map and \(\varphi_t * u\) converges to \(u\) in \(W^{s,p}(Q^m; \R^\nu)\), by the property of continuity of maps in \(W^{s, p} \cap L^\infty\) under composition \citelist{\cite{Brezis-Mironescu-2001}*{Theorem~1.1'} \cite{Mazya-Shaposhnikova}*{Theorem}}, for every \(\xi\in B^{\nu}_{\delta}\),
\begin{equation}
\label{eqStrongDensityConvergenceConvolution}
\lim_{t \to 0}{\norm{\eta_{\xi} \circ (\varphi_{t} * u) - \eta_{\xi} \circ u}_{W^{s,p}(Q^m)}} = 0.
\end{equation}

In view of \eqref{eqStrongDensityTriangleInequality}, we need to control the quantity
\[
\norm{\eta_{\xi}\circ u - u}_{W^{s,p}(Q^m)}
\]
for some suitable \(\xi \in B_\delta^\nu\).
We start with the following:

\begin{Claim}
There exists a nonnegative function \(F \in L^1(Q^m)\) depending on  \(s\), \(p\), \(m\) and \(u\) such that for every \(\xi \in B_\delta^\nu\),
\[
\norm{\eta_{\xi}\circ u - u}_{W^{s,p}(Q^m)}^p
 \le [\eta_\xi]_{C^{k+1}(\R^\nu)}^{\sigma p} [\eta_\xi]_{C^{k}(\R^\nu)}^{(1-\sigma)p} \int\limits_{\{\eta_{\xi}\circ u \neq u\}} F.
\]
\end{Claim}

\begin{proof}[Proof of the claim]
By definition of the \(W^{s, p}\) norm,
\begin{multline*}
 \norm{\eta_{\xi}\circ u - u}_{W^{s,p}(Q^m)}\\
 = \sum_{j=0}^{k} \norm{ D^j(\eta_{\xi}\circ u) -  D^j u}_{L^{p}(Q^m)} +  \norm{D^{s,p} (\eta_{\xi}\circ u - u)}_{L^{p}(Q^m)};
\end{multline*}
when \(s\) is an integer, we disregard the last term.

Since the map \(u\) is bounded,
\begin{equation}
\label{eqEstimateLp}
\norm{\eta_{ \xi}\circ u - u}_{L^{p}(Q^m)} 
= \norm{\eta_{\xi}\circ u - u}_{L^{p}(\{\eta_{\xi}\circ u \neq u\})} \le {\NewConstant} \cH^m(\{\eta_{\xi}\circ u \neq u\}).
\end{equation}
Moreover, for every \(j \in \{1, \dots, k\}\),
\[
\begin{split}
\norm{ D^j(\eta_{\xi}\circ u) - D^j u}_{L^{p}(Q^m)} 
& = 
\norm{ D^j(\eta_{\xi}\circ u) - D^j u}_{L^{p}(\{\eta_{\xi}\circ u \neq u\})}\\
& \le 
\norm{ D^j(\eta_{\xi}\circ u)}_{L^{p}(\{\eta_{\xi}\circ u \neq u\})} + 
\norm{D^j u}_{L^{p}(\{\eta_{\xi}\circ u \neq u\})}.
\end{split}
\]
Since \(u \in W^{s, p}(Q^m; \R^\nu) \cap L^\infty(Q^m; \R^\nu)\), by the Gagliardo-Nirenberg interpolation inequality \cite{Brezis-Mironescu-2001}*{Corollary~3.2}, \(u \in W^{1, sp}(Q^m; \R^\nu)\).
By Proposition~\ref{propositionPointwiseEstimate}, there exists a function \(G_j \in L^p(Q^m)\) independent of \(\eta_\xi\) such that
\begin{equation}\label{eqEstimateWkp}
\norm{ D^j(\eta_{\xi}\circ u) - D^j u}_{L^{p}(Q^m)} 
\le 
[\eta_{\xi}]_{C^j(\R^\nu)} \norm{G_j}_{L^p(\{\eta_{\xi}\circ u \neq u\})} + \norm{D^j u}_{L^{p}(\{\eta_{\xi}\circ u \neq u\})}.
\end{equation}
If \(s\) is non-integer, then
\[
\begin{split}
\norm{D^{s,p} (\eta_{\xi}\circ u - u)}_{L^{p}(Q^m)}
& \leq 
2^{1/p} \norm{D^{s,p} (\eta_{\xi}\circ u - u)}_{L^{p}(\{\eta_{\xi}\circ u \neq u\})}\\
& \le 
2^{1/p} \big(\norm{D^{s,p} (\eta_{\xi}\circ u)}_{L^{p}(\{\eta_{\xi}\circ u \neq u\})} +  \norm{D^{s,p} u}_{L^{p}(\{\eta_{\xi}\circ u \neq u\})} \big). 
\end{split}
\]
By Proposition~\ref{propositionPointwiseEstimateWsp}, there exists  \(H \in L^p(Q^m)\) independent of \(\eta_\xi\) such that
\begin{multline}
\label{eqEstimateWsp}
\norm{D^{s,p} (\eta_{\xi}\circ u - u)}_{L^{p}(Q^m)}\\
\le 
2^{1/p} \bigl( [\eta_{\xi}]_{C^{k+1}(\R^\nu)}^{\sigma} [\eta_{\xi}]_{C^k(\R^\nu)}^{1 - \sigma} \norm{H}_{L^{p}(\{\eta_{\xi}\circ u \neq u\})} +  \norm{D^{s,p} u}_{L^{p}(\{\eta_{\xi}\circ u \neq u\})}\bigr). 
\end{multline}
Combining estimates \eqref{eqEstimateLp}, \eqref{eqEstimateWkp} and \eqref{eqEstimateWsp}, we conclude that
\[
 \norm{\eta_{\xi}\circ u - u}_{W^{s,p}(Q^m)}^p
 \le [\eta_{\xi}]_{C^{k+1}(\R^\nu)}^{\sigma p} [\eta_{\xi}]_{C^k(\R^\nu)}^{(1 - \sigma)p} \int\limits_{\{\eta_{\xi}\circ u \neq u\}} F,
\]
with
\[
F = \Constant\bigg(1 + \sum_{j = 1}^k (G_j^p + \abs{D^j u}^p) + H^p + (D^{s, p}u)^p \bigg).
\]
This proves the claim.
\end{proof}

Let \(K \subset N^n\) be a compact set such that for every \(x \in N^n\setminus K\),
\begin{equation}
\label{eqStrongDensityEta}
\eta(x)=x.
\end{equation}
If \(x \in N^n\) is such that \(\eta_{\xi}(u(x)) \neq u(x)\) for some \(\xi \in B_\delta^\nu\), then 
\[
P_\xi(u(x)) = \Pi(u(x) - \xi) \in K,
\]
whence \(x \in u^{-1}(\Pi^{-1}(K) + \xi)\). 
In other words, for every \(\xi \in B_\delta^\nu\),
\begin{equation}
\label{eqStrongDensityInclusion}
\{\eta_{\xi}\circ u \neq u\} \subset u^{-1}(\Pi^{-1}(K) + \xi).
\end{equation}
Thus, from the previous claim,
\[
\begin{split}
\norm{\eta_{\xi}\circ u - u}_{W^{s,p}(Q^m)}^p
& \le [\eta_{\xi}]_{C^{k+1}(\R^\nu)}^{\sigma p} [\eta_{\xi}]_{C^k(\R^\nu)}^{(1 - \sigma)p} \int\limits_{u^{-1}(\Pi^{-1}(K) + \xi)} F\\
& \le \Constant [\eta]_{C^{k+1}(\R^\nu)}^{\sigma p} [\eta]_{C^k(\R^\nu)}^{(1 - \sigma)p} \int\limits_{u^{-1}(\Pi^{-1}(K) + \xi)} F,
 \end{split}
\]
for some constant \(\SameConstant > 0\) independent of \(\xi\).
By Lemma~\ref{lemmefubini}, we get
\[
\int\limits_{B_\delta^\nu} \norm{\eta_{\xi}\circ u - u}_{W^{s,p}(Q^m)}^p \dif \xi
\le \SameConstant [\eta]_{C^{k+1}(\R^\nu)}^{\sigma p} [\eta]_{C^k(\R^\nu)}^{(1 - \sigma)p} \cH^{\nu}(\Pi^{-1}(K)) \norm{F}_{L^1(Q^m)}.
\]
Since
\[
\cH^{\nu}(\Pi^{-1}(K)) \le \Constant \cH^n(K),
\]
we conclude that
\[
\int\limits_{B_\delta^\nu} \norm{\eta_{\xi}\circ u - u}_{W^{s,p}(Q^m)}^p \dif \xi
\le  \Constant [\eta]_{C^{k+1}(\R^\nu)}^{\sigma p} [\eta]_{C^k(\R^\nu)}^{(1 - \sigma)p}  \cH^n(K) \norm{F}_{L^1(Q^m)}.
\]

Let \(0 < \epsilon \le 1\).
Since \(N^n\) is \(\floor{sp}\) simply connected, by Proposition~\ref{propositionSmoothProjection} there exists a smooth map \(\eta\)  satisfying 
\eqref{eqStrongDensityEta} for some compact set \(K \subset N^n\) such that
\begin{equation}\label{equationmeasureK}
\mathcal{H}^n(K) \leq {\Constant} \epsilon^{\floor{sp}+1}
\end{equation}
and for every \(j \in \{1, \dots, k+1\}\),
\begin{equation*}
\norm{D^j \eta}_{L^\infty(\R^\nu)} \le \frac{{\Constant}}{\epsilon^j}.
\end{equation*}
Thus,
\begin{equation}\label{estimatewspxi}
\int\limits_{B_\delta^\nu} \norm{\eta_{\xi}\circ u - u}_{W^{s,p}(Q^m)}^p \dif \xi
\le  {\Constant} \epsilon^{\floor{sp}+1 - sp}.
\end{equation}

Since \(sp < \floor{sp} + 1\), we can thus find a smooth map \(\eta = \eta_\epsilon\) and \(\xi  = \xi_\epsilon \in B^{\nu}_\delta\) such that 
\[
\lim_{\epsilon \to 0}{\norm{\eta_{\epsilon, \xi_\epsilon} \circ u - u}_{W^{s,p}(Q^m)}}= 0.
\]
By \eqref{eqStrongDensityConvergenceConvolution}, for every \(0 < \epsilon \le 1\) there exists \(0 < t_\epsilon \le \gamma\) such that
\[
\lim_{\epsilon \to 0}{\norm{\eta_{\epsilon, \xi_\epsilon} \circ (\varphi_{t_\epsilon} * u) - \eta_{\epsilon, \xi_\epsilon} \circ u}_{W^{s,p}(Q^m)}}= 0.
\]
Thus, by the triangle inequality \eqref{eqStrongDensityTriangleInequality},
\[
\lim_{\epsilon \to 0}{\norm{\eta_{\epsilon, \xi_\epsilon} \circ (\varphi_{t_\epsilon} * u) - u}_{W^{s,p}(Q^m)}}= 0.
\]
This completes the proof of Theorem~\ref{deuxiemetheorem}.
\end{proof}


\section{Weak density}

We prove a more precise version of Theorem~\ref{theoremweakdensity}:

\begin{theorem}\label{theoremweakdensitybis}
Let \(s \geq 1\). 
If  \(sp < m\) is such that \(sp \in \N\) and if \(N^n\) is \(sp-1\) simply connected, then for every \(u \in W^{s, p}(Q^m; N^n)\) there exists a sequence \((u_i)_{i \in \N} \) in \(C^{\infty}(\overline Q^m; N^n)\) such that
\begin{enumerate}[\((i)\)]
\item \((u_i)_{i \in \N}\) converges in measure to \(u\),
\item for every \(j \in \{1, \dots, k\}\), \((D^j u_i)_{i \in \N}\) converges in measure to \(D^j u\),
\item for every \(i \in \N\), 
\[
\norm{u_i}_{W^{s, p}(Q^m)} \le C,
\]
for some constant \(C > 0\) depending on \(s\), \(p\), \(m\), \(\norm{u}_{W^{s, p}(Q^m)}\) and \(N^n\).
\end{enumerate}
\end{theorem}

\begin{proof}
We explain what should be changed in the proof of Theorem~\ref{deuxiemetheorem}.
Since \(N^n\) is now merely \(sp-1\) simply connected, the map \(\eta\) may be chosen so that \(\eta(x) = x\) on \(N^n \setminus K\), where the compact set \(K\) satisfies 
\begin{equation}
\label{eqWeakDensityMeasureK}
\mathcal{H}^n(K) \leq \NewConstant \epsilon^{sp}.
\end{equation}
By inclusion~\eqref{eqStrongDensityInclusion}, by Lemma~\ref{lemmefubini} and by property~\eqref{eqWeakDensityMeasureK},
\begin{equation*}
\int\limits_{B_\delta^\nu} \cH^m\bigl(\{\eta_{\xi}\circ u \neq u\}\bigr) \dif \xi
\le \int\limits_{B_\delta^\nu} \cH^m\bigl(u^{-1}(\Pi^{-1}(K) + \xi)\bigr) \dif \xi
\le \Constant \cH^n(K) \cH^m(Q^m) 
\le \Constant \epsilon^{sp}.
\end{equation*}
Replacing \eqref{equationmeasureK} by \eqref{eqWeakDensityMeasureK}, estimate \eqref{estimatewspxi} becomes
\begin{equation*}
\int\limits_{B_\delta^\nu} \norm{\eta_{\xi}\circ u - u}_{W^{s,p}(Q^m)}^p \dif \xi
\le  \Constant .
\end{equation*}
Thus, for every \(0 < \epsilon \le 1\) there exists a smooth map \(\eta = \eta_\epsilon\) and \(\xi  = \xi_\epsilon \in B^{\nu}_\delta\)  such that
\[
\cH^m \bigl(\{\eta_{\epsilon, \xi_\epsilon} \circ u \neq u\}\bigr) \le \Constant \epsilon^{sp}
\]
and
\[
\norm{\eta_{\epsilon, \xi_\epsilon}\circ u - u}_{W^{s,p}(Q^m)}^p
\le \Constant.
\]

As in the proof of Theorem~\ref{deuxiemetheorem}, for every \(0 < \epsilon \le 1\) we find \(0 < t_\epsilon \le \gamma\) such that
\begin{equation}
\label{eqWeakDensityConvergenceConvolution}
\lim_{\epsilon \to 0}{\norm{\eta_{\epsilon, \xi_\epsilon} \circ (\varphi_{t_\epsilon} * u) - \eta_{\epsilon, \xi_\epsilon} \circ u}_{W^{s,p}(Q^m)}}= 0.
\end{equation}
Thus, by the triangle inequality,
\[
\norm{\eta_{\epsilon, \xi_\epsilon}\circ (\varphi_{t_\epsilon} * u)}_{W^{s,p}(Q^m)}
\le \Constant.
\]
Note that \(\eta_{\epsilon, \xi_{\epsilon}} \circ u\) and \(u\) as well as their derivatives up to order \(k\) coincide almost everywhere on \(\{\eta_{\xi_\epsilon, \epsilon} \circ u = u\}\). 
Combining
\[
\lim_{\epsilon \to 0}{\cH^m \bigl(\{\eta_{\epsilon, \xi_\epsilon}\circ u \neq u\}\bigr)} = 0
\]
and \eqref{eqWeakDensityConvergenceConvolution}, we deduce the convergence in measure of \(\eta_{\epsilon, \xi_{\epsilon}} \circ (\varphi_{t_\epsilon} * u)\) and its derivatives as \(\epsilon\) tends to zero.
This completes the proof of Theorem~\ref{theoremweakdensity}.
\end{proof}



\section*{Acknowledgments}

The second (ACP) and third (JVS) authors were supported by the Fonds de la Recherche scientifique---FNRS.



\begin{bibdiv}

\begin{biblist}

\bib{Adams_1975}{book}{
   author={Adams, Robert A.},
   title={Sobolev spaces},
   note={Pure and Applied Mathematics, Vol. 65},
   publisher={Academic Press, New York-London},
   date={1975},
}



\bib{Bethuel-1990}{article}{
   author={Bethuel, Fabrice},
   title={A characterization of maps in $H^1(B^3,S^2)$ which can be
   approximated by smooth maps},
   journal={Ann. Inst. H. Poincar\'e Anal. Non Lin\'eaire},
   volume={7},
   date={1990},
   pages={269--286},
}

\bib{Bethuel}{article}{
   author={Bethuel, Fabrice},
   title={The approximation problem for Sobolev maps between two manifolds},
   journal={Acta Math.},
   volume={167},
   date={1991},
   pages={153--206},
}

\bib{Bethuel-1995}{article}{
   author={Bethuel, Fabrice},
   title={Approximations in trace spaces defined between manifolds},
   journal={Nonlinear Anal.},
   volume={24},
   date={1995},
   pages={121--130},
}
	


\bib{Bethuel-Zheng}{article}{
   author={Bethuel, Fabrice},
   author={Zheng, Xiao Min},
   title={Density of smooth functions between two manifolds in Sobolev
   spaces},
   journal={J. Funct. Anal.},
   volume={80},
   date={1988},
   pages={60--75},
}


\bib{Bousquet-Ponce-VanSchaftingen}{article}{
   author={Bousquet, Pierre},
   author={Ponce, Augusto C.},
   author={Van Schaftingen, Jean},
   title={Strong density for higher order Sobolev spaces into compact manifolds},
    status={submitted paper},
}



\bib{Brezis-Mironescu-2001}{article}{
   author={Brezis, Ha{\"{\i}}m},
   author={Mironescu, Petru},
   title={Gagliardo-Nirenberg, composition and products in fractional
   Sobolev spaces},
   journal={J. Evol. Equ.},
   volume={1},
   date={2001},
   pages={387--404},
}

\bib{Brezis-Mironescu}{article}{
   author={Brezis, Haim},
   author={Mironescu, Petru},
   status={in preparation},
}



\bib{Escobedo}{article}{
   author={Escobedo, Miguel},
   title={Some remarks on the density of regular mappings in Sobolev classes
   of $S^M$-valued functions},
   journal={Rev. Mat. Univ. Complut. Madrid},
   volume={1},
   date={1988},
   pages={127--144},
}

\bib{Federer-Fleming}{article}{
   author={Federer, Herbert},
   author={Fleming, Wendell H.},
   title={Normal and integral currents},
   journal={Ann. of Math. (2)},
   volume={72},
   date={1960},
   pages={458--520},
}




\bib{Gagliardo}{article}{
   author={Gagliardo, Emilio},
   title={Caratterizzazioni delle tracce sulla frontiera relative ad alcune
   classi di funzioni in $n$ variabili},
   journal={Rend. Sem. Mat. Univ. Padova},
   volume={27},
   date={1957},
   pages={284--305},
}

\bib{Gagliardo-1959}{article}{
   author={Gagliardo, Emilio},
   title={Ulteriori propriet\`a di alcune classi di funzioni in pi\`u
   variabili},
   journal={Ricerche Mat.},
   volume={8},
   date={1959},
   pages={24--51},
}

\bib{Gastel-Nerf}{article}{
   author={Gastel, Andreas},
   author={Nerf, Andreas J.},
   title={Density of smooth maps in $W^{k,p}(M,N)$ for a close to
   critical domain dimension},
   journal={Ann. Global Anal. Geom.},
   volume={39},
   date={2011},
   pages={107--129},
}

\bib{Hajlasz}{article}{
   author={Haj{\l}asz, Piotr},
   title={Approximation of Sobolev mappings},
   journal={Nonlinear Anal.},
   volume={22},
   date={1994},
   pages={1579--1591},
}



\bib{Hang}{article}{
   author={Hang, Fengbo},
   title={Density problems for $W^{1,1}(M,N)$},
   journal={Comm. Pure Appl. Math.},
   volume={55},
   date={2002},
   pages={937--947},
}

\bib{Hang-Lin}{article}{
   author={Hang, Fengbo},
   author={Lin, Fanghua},
   title={Topology of Sobolev mappings. II},
   journal={Acta Math.},
   volume={191},
   date={2003},
   pages={55--107},
}

\bib{Hang-Lin-2003}{article}{
   author={Hang, Fengbo},
   author={Lin, Fanghua},
   title={Topology of Sobolev mappings. III},
   journal={Comm. Pure Appl. Math.},
   volume={56},
   date={2003},
   pages={1383--1415},
}


\bib{Hardt-Kinderlehrer-Lin}{article}{
   author={Hardt, R.},
   author={Kinderlehrer, D.},
   author={Lin, Fang-Hua},
   title={Stable defects of minimizers of constrained variational
   principles},
   journal={Ann. Inst. H. Poincar\'e Anal. Non Lin\'eaire},
   volume={5},
   date={1988},
   pages={297--322},
}



\bib{Hedberg}{article}{
   author={Hedberg, Lars Inge},
   title={On certain convolution inequalities},
   journal={Proc. Amer. Math. Soc.},
   volume={36},
   date={1972},
   pages={505--510},
}

\bib{Mazya}{book}{
   author={Maz\cprime ya, Vladimir},
   title={Sobolev spaces with applications to elliptic partial differential
   equations},
   series={Grundlehren der Mathematischen Wissenschaften},
   volume={342},
   publisher={Springer},
   place={Heidelberg},
   date={2011},
}


\bib{Mazya-Shaposhnikova}{article}{
   author={Maz\cprime ya, V.},
   author={Shaposhnikova, T.},
   title={An elementary proof of the Brezis and Mironescu theorem on the
   composition operator in fractional Sobolev spaces},
   journal={J. Evol. Equ.},
   volume={2},
   date={2002},
   pages={113--125},
}

\bib{Mironescu}{article}{
   author={Mironescu, Petru},
   title={Sobolev maps on manifolds: degree, approximation, lifting},
   conference={
      title={Perspectives in nonlinear partial differential equations},
   },
   book={
      editor={Berestycki, Henri},
      editor={Bertsch, Michiel},
      editor={Browder, Felix E.},
      editor={Nirenberg, Louis},
      editor={Peletier, Lambertus A.},
      editor={V{\'e}ron, Laurent},
      series={Contemp. Math.},
      volume={446},
      publisher={Amer. Math. Soc.},
      place={Providence, RI},
   },
   date={2007},
   pages={413--436},
   note={In honor of Ha\"\i m Brezis},
}

\bib{Mucci}{article}{
   author={Mucci, Domenico},
   title={Strong density results in trace spaces of maps between manifolds},
   journal={Manuscripta Math.},
   volume={128},
   date={2009},
   pages={421--441},
}
	
	
\bib{Nirenberg-1959}{article}{
   author={Nirenberg, Louis},
   title={On elliptic partial differential equations},
   journal={Ann. Scuola Norm. Sup. Pisa (3)},
   volume={13},
   date={1959},
   pages={115--162},
}

\bib{Oru}{thesis}{
  author = {Oru, F.},
  type = {Th\`ese de doctorat},
  organization = {\'Ecole Normale Sup\'erieure de Cachan},
  title = {R\^ole des oscillations dans quelques probl\`emes d'analyse non lin\'eaire},
  date = {1998},
}


\bib{Pakzad}{article}{
   author={Pakzad, Mohammad Reza},
   title={Weak density of smooth maps in $W^{1,1}(M,N)$ for non-abelian
   $\pi_1(N)$},
   journal={Ann. Global Anal. Geom.},
   volume={23},
   date={2003},
   pages={1--12},
}


\bib{Pakzad-Riviere}{article}{
   author={Pakzad, M. R.},
   author={Rivi{\`e}re, T.},
   title={Weak density of smooth maps for the Dirichlet energy between
   manifolds},
   journal={Geom. Funct. Anal.},
   volume={13},
   date={2003},
   pages={223--257},
}

\bib{Riviere}{article}{
   author={Rivi{\`e}re, Tristan},
   title={Dense subsets of $H^{1/2}(S^2,S^1)$},
   journal={Ann. Global Anal. Geom.},
   volume={18},
   date={2000},
   pages={517--528},
}
	


\bib{RunstSickel}{book}{
   author={Runst, Thomas},
   author={Sickel, Winfried},
   title={Sobolev spaces of fractional order, Nemytskij operators, and
   nonlinear partial differential equations},
   series={de Gruyter Series in Nonlinear Analysis and Applications},
   volume={3},
   publisher={Walter de Gruyter \& Co.},
   place={Berlin},
   date={1996},
}
\bib{SchoenUhlenbeck}{article}{
   author={Schoen, Richard},
   author={Uhlenbeck, Karen},
   title={Boundary regularity and the Dirichlet problem for harmonic maps},
   journal={J. Differential Geom.},
   volume={18},
   date={1983},
   pages={253--268},
}


\bib{Stein}{book}{
   author={Stein, Elias M.},
   title={Singular integrals and differentiability properties of functions},
   series={Princeton Mathematical Series, No. 30},
   publisher={Princeton University Press},
   place={Princeton, N.J.},
   date={1970},
}

\bib{White}{article}{
   author={White, Brian},
   title={Infima of energy functionals in homotopy classes of mappings},
   journal={J. Differential Geom.},
   volume={23},
   date={1986},
   pages={127--142},
}

\end{biblist}
 
\end{bibdiv}
\end{document}